\documentclass[11pt]{amsart}
\usepackage{amsmath,amssymb,amsthm,array,color,multirow,pdflscape,graphicx,pigpen,stmaryrd,comment}
\usepackage[all]{xypic}
\usepackage[normalem]{ulem}

\usepackage[all,arc,2cell]{xy}
\UseAllTwocells

\usepackage{lmodern}

\usepackage[margin=1.5in]{geometry}

\newlength{\storeparskip}
\usepackage{xcolor}
\definecolor{darkgreen}{rgb}{0,0.30,0} 
\definecolor{darkred}{rgb}{0.75,0,0}
\definecolor{darkblue}{rgb}{0,0,0.6} 
\usepackage[pdfborder=0,pagebackref,colorlinks,citecolor=darkgreen,linkcolor=darkgreen,urlcolor=darkblue]{hyperref}
\renewcommand*{\backref}[1]{}
\renewcommand*{\backrefalt}[4]{({%
    \ifcase #1 Not cited.%
          \or On p.~#2%
          \else On pp.~#2%
    \fi%
    })}

\def\makeautorefname#1#2{\expandafter\def\csname#1autorefname\endcsname{#2}}
\makeautorefname{equation}{Equation}%
\makeautorefname{footnote}{footnote}%
\makeautorefname{item}{item}%
\makeautorefname{figure}{Figure}%
\makeautorefname{table}{Table}%
\makeautorefname{part}{Part}%
\makeautorefname{appendix}{Appendix}%
\makeautorefname{chapter}{Chapter}%
\makeautorefname{section}{Section}%
\makeautorefname{subsection}{Section}%
\makeautorefname{subsubsection}{Section}%
\makeautorefname{paragraph}{Paragraph}%
\makeautorefname{subparagraph}{Paragraph}%
\makeautorefname{theorem}{Theorem}%
\makeautorefname{thm}{Theorem}%
\makeautorefname{addm}{Addendum}%
\makeautorefname{mainthm}{Main theorem}%
\makeautorefname{corollary}{Corollary}%
\makeautorefname{cor}{Corollary}%
\makeautorefname{lemma}{Lemma}%
\makeautorefname{lem}{Lemma}%
\makeautorefname{sublemma}{Sublemma}%
\makeautorefname{sublem}{Sublemma}%
\makeautorefname{subl}{Sublemma}%
\makeautorefname{prop}{Proposition}%
\makeautorefname{property}{Property}
\makeautorefname{pro}{Property}
\makeautorefname{sch}{Scholium}%
\makeautorefname{step}{Step}%
\makeautorefname{conject}{Conjecture}%
\makeautorefname{conj}{Conjecture}%
\makeautorefname{questn}{Question}
\makeautorefname{quest}{Question}
\makeautorefname{qn}{Question}
\makeautorefname{definition}{Definition}%
\makeautorefname{defn}{Definition}%
\makeautorefname{defi}{Definition}%
\makeautorefname{def}{Definition}%
\makeautorefname{dfn}{Definition}%
\makeautorefname{notation}{Notation}
\makeautorefname{notn}{Notation}
\makeautorefname{rem}{Remark}%
\makeautorefname{rems}{Remarks}%
\makeautorefname{rmk}{Remark}%
\makeautorefname{remark}{Remark}%
\makeautorefname{rk}{Remark}%
\makeautorefname{remarks}{Remarks}%
\makeautorefname{rems}{Remarks}%
\makeautorefname{rmks}{Remarks}%
\makeautorefname{rks}{Remarks}%
\makeautorefname{example}{Example}%
\makeautorefname{examp}{Example}%
\makeautorefname{exmp}{Example}%
\makeautorefname{exam}{Example}%
\makeautorefname{exa}{Example}%
\makeautorefname{axiom}{Axiom}%
\makeautorefname{axi}{Axiom}%
\makeautorefname{ax}{Axiom}%
\makeautorefname{case}{Case}%
\makeautorefname{claim}{Claim}%
\makeautorefname{clm}{Claim}%
\makeautorefname{assumpt}{Assumption}%
\makeautorefname{asses}{Assumptions}%
\makeautorefname{conclusion}{Conclusion}%
\makeautorefname{concl}{Conclusion}%
\makeautorefname{conc}{Conclusion}%
\makeautorefname{cond}{Condition}%
\makeautorefname{const}{Construction}%
\makeautorefname{con}{Construction}%
\makeautorefname{criterion}{Criterion}%
\makeautorefname{criter}{Criterion}%
\makeautorefname{crit}{Criterion}%
\makeautorefname{exercise}{Exercise}%
\makeautorefname{exer}{Exercise}%
\makeautorefname{exe}{Exercise}%
\makeautorefname{problem}{Problem}%
\makeautorefname{problm}{Problem}%
\makeautorefname{prob}{Problem}%
\makeautorefname{prob}{Problem}%
\makeautorefname{soln}{Solution}%
\makeautorefname{sol}{Solution}%
\makeautorefname{sum}{Summary}%
\makeautorefname{oper}{Operation}%
\makeautorefname{obs}{Observation}%
\makeautorefname{ob}{Observation}%
\makeautorefname{conv}{Convention}%
\makeautorefname{cvn}{Convention}%
\makeautorefname{warn}{Warning}%
\makeautorefname{note}{Note}%
\makeautorefname{fact}{Fact}%
\makeautorefname{ouch0}{Counterexample}%
\newtheorem{thm}{Theorem}[section]
\newtheorem{cor}{Corollary}[section]
\newtheorem{prop}{Proposition}[section]
\newtheorem{lem}{Lemma}[section]

\theoremstyle{definition}

\newtheorem{df}{Definition}[section]

\newtheorem{obs}{Observation}[section]
\newtheorem{rem}{Remark}[section]
\newtheorem{remark}{Remark}[section]

\makeatletter
\let\c@cor=\c@thm
\let\c@prop=\c@thm
\let\c@lem=\c@thm
\let\c@conj=\c@thm
\let\c@defn=\c@thm
\let\c@df=\c@thm
\let\c@exmp=\c@thm
\let\c@rem=\c@thm
\let\c@sch=\c@thm
\let\c@equation\c@thm
\makeatother

\newcommand{\cat}[1]{\textup{\textbf{{#1}}}}

\newcommand{\End}{\textup{End}}

\newcommand{\id}{\textup{id}}

\newcommand{\ra}{\longrightarrow}

\newcommand{\sma}{\wedge}
\newcommand{\barsmash}{\,\overline\wedge\,}
\newcommand{\ti}{\mathcal{E}}

\newcommand{\congar}{\overset\cong\longrightarrow}

\newcommand{\mc}{\mathcal}

\newcommand{\xra}{\xrightarrow}

\newcommand{\uda}{\rotatebox[origin=c]{180}{\textbf{A}}}

\newcommand{\sC}{\mathcal{C}}
\newcommand{\sD}{\mathcal{D}}
\newcommand{\sA}{\mathcal{A}}
\newcommand{\sB}{\mathcal{B}}
\newcommand{\sO}{\mathcal{O}}

\newcommand{\GB}{\mathcal{GB}}

\newcommand{\bK}{\textbf{K}}
\newcommand{\bA}{\textbf{A}}

\newcommand{\tG}{\mathcal{E}G}
\newcommand{\tH}{\mathcal{E}H}

\newcommand{\po}{\ar@{}[dr]|{\text{\pigpenfont R}}}
\newcommand{\pb}{\ar@{}[dr]|{\text{\pigpenfont J}}}

\DeclareMathOperator{\Cat}{\mathrm{Cat}}

\newcommand{\sbt}{\,\begin{picture}(-1,1)(0.5,-1)\circle*{1.8}\end{picture}\hspace{.05cm}}

\title{Equivariant $A$-theory}
\author{Cary Malkiewich}
\address{Department of Mathematics, Binghamton University}
\email{malkiewich@math.binghamton.edu}
\author{Mona Merling}
\address{Department of Mathematics, The University of Pennsylvania}
\email{mmerling@math.upenn.edu}

\begin{document}

\begin{abstract}
We give a new construction of the equivariant $K$-theory of group actions (cf. Barwick et al.), producing an infinite loop $G$-space for each Waldhausen category with $G$-action, for a finite group $G$. On the category $R(X)$ of retractive spaces over a $G$-space $X$, this produces an equivariant lift of Waldhausen's functor $\bA(X)$, and we show that the $H$-fixed points are the bivariant $A$-theory of the fibration $X_{hH}\to BH$. We then use the framework of spectral Mackey functors to produce a second equivariant refinement $\bA_G(X)$ whose fixed points have tom Dieck type splittings. We expect this second definition to be suitable for an equivariant generalization of the parametrized $h$-cobordism theorem.

\end{abstract}

\maketitle

\begingroup%
\setlength{\parskip}{\storeparskip}
\tableofcontents
\endgroup%

\section{Introduction}

Waldhausen's celebrated $\bA(X)$ construction, and the ``parametrized $h$-cobordism'' theorem relating it to the space of $h$-cobordisms $\mathcal H^\infty(X)$ on $X$, provides a critical link in the chain of homotopy-theoretic constructions relating the behavior of compact manifolds to that of their underlying homotopy types \cite{waldhausen1978alg} \cite{waldhausennew}. While the $L$-theory assembly map provides the primary invariant that distinguishes the closed manifolds in a given homotopy type, $\bA(X)$ provides the secondary information that accesses the diffeomorphism and homeomorphism groups in a stable range \cite{ww}. And in the case of compact manifolds up to stabilization, $\bA(X)$ accounts for the entire difference between the manifold and its underlying homotopy type with tangent information \cite{dww}. As a consequence, calculations of $\bA(X)$ have immediate consequences for the automorphism groups of high-dimensional closed manifolds, and of compact manifolds up to stabilization.

When the manifolds in question have an action by a group $G$, there is a similar line of attack for understanding the equivariant homeomorphisms and diffeomorphisms. One expects to replace $\mathcal H^\infty(X)$ with an appropriate space $\mathcal H^\infty(X)^G$ of $G$-isovariant $h$-cobordisms on $X$, stabilized with respect to representations of $G$. The connected components of such a space would be expected to coincide with the equivariant Whitehead group of \cite{luck}, which splits as
\begin{equation}\label{eq:pi0_splitting}
Wh^G(X) \cong \bigoplus_{(H) \leq G} Wh(X^H_{hWH})
\end{equation}
where $(H) \leq G$ denotes conjugacy classes of subgroups. This splitting is reminiscent of the tom Dieck splitting for genuine $G$-suspension spectra
\[ (\Sigma^\infty_G X_+)^G \cong \bigvee_{(H) \leq G} \Sigma^\infty_+ X^H_{hWH} \]
and suggests that the variant of $A$-theory most directly applicable to manifolds will in fact be a genuine $G$-spectrum, whose fixed points have a similar splitting.

In this paper we begin to realize this conjectural framework. We define an equivariant generalization $\bA_G(X)$ of Waldhausen's $A$-theory functor, when $X$ is a space with an action by a finite group $G$, whose fixed points have the desired tom Dieck style splitting.
\begin{thm}[\autoref{agx_exists2}]\label{thm:agx_exists}
For $G$ a finite group, there exists a functor $\bA_G$ from $G$-spaces to genuine $G$-spectra with fixed points
\[ \bA_G(X)^G \simeq \prod_{(H) \leq G} \bA(X^H_{hWH}), \]
and a similar formula for the fixed points of each subgroup $H$.
\end{thm}

To be more specific, the fixed points are the $K$-theory of the category $R^G_{hf}(X)$ of finite retractive $G$-cell complexes over $X$, with equivariant weak homotopy equivalences between them. The splitting of this $K$-theory is a known consequence of the additivity theorem, and an explicit proof appears both in \cite{wojciech} and in earlier unpublished work by John Rognes from the early 1990s. In fact, this earlier work by Rognes seems to be the first place where the spectrum $K(R^G_{hf}(X))$ was studied, and it was motivated by a possible variant of the Segal conjecture for $A$-theory.
%
%


In a subsequent paper, we plan to explain how $\bA_G(X)$ fits into a genuinely $G$-equivariant generalization of Waldhausen's parametrized $h$-cobordism theorem. The argument we have in mind draws significantly from an analysis of the fixed points of our $\bA_G(X)$ carried out by Badzioch and Dorabia\l{}a \cite{wojciech}, and a forthcoming result of Goodwillie and Igusa that defines $\mc H^\infty(X)^G$ and gives a splitting that recovers \eqref{eq:pi0_splitting}. We emphasize that lifting these theorems to genuine $G$-spectra permits the tools of equivariant stable homotopy theory to be applied to the calculation of $\mc H^\infty(X)^G$, in addition to the linearization and trace techniques that have been used so heavily in the nonequivariant case.

Most of the work in this paper is concerned with constructing equivariant spectra out of category-theoretic data. One approach is to generalize classical delooping constructions such as the operadic machine of May \cite{MayGeo} or  the $\Gamma$-space machine of Segal \cite{segal} to allow for deloopings by representations of $G$. Using the equivariant generalization of the operadic infinite loop space machine from \cite{GM3}, we show how this approach generalizes to deloop Waldhausen $G$-categories.

The theory of Waldhausen categories with $G$-action is subtle. Even when the $G$-action is through exact functors, the fixed points of such a category do not necessarily have Waldhausen structure (\autoref{waldfixedpts}). Define $\tG$ be the category with objects the elements of $G$ and precisely one morphism between any two objects, whose classifying space is $EG$. Let $\Cat(\tG, \sC)$ be the category of all functors and all natural transformations with $G$ acting by conjugation; we define the \emph{homotopy fixed points} $\sC^{hG}$ of a $G$-category $\sC$ as the fixed point category $\Cat(\tG, \sC)^G$, and we explain in \S\autoref{waldhausen gcat} how this category does have a Waldhausen structure.

The ``equivariant $K$-theory of group actions'' of Barwick, Glasman, and Shah produces a genuine $G$-spectrum (using the framework of \cite{Gmonster}) whose $H$-fixed points are $K(\mc C^{hH})$ \cite[\S 8]{Gmonster2}. We complement this with a result that shows the $G$-space $|\Cat(\tG, \sC)|$ may be directly, equivariantly delooped.

\begin{thm}[\autoref{inf loop} and \autoref{fixed_points_agree}]\label{thm:can_deloop_gwalds}
If $\mc C$ is a Waldhausen $G$-category then the $K$-theory space defined as $K_G(\sC):=\Omega |w\mc S_{\sbt} \Cat(\tG, \sC)|$, where $\mc S_{\sbt} $ is Waldhausen's construction from \cite{waldhausen},  is an equivariant infinite loop space. The $H$-fixed points of the resulting $\Omega$-$G$-spectrum are equivalent to the $K$-theory of the Waldhausen category $\sC^{hH}$ for every subgroup $H$.
\end{thm}


The downside of this approach is that one does not have much freedom to modify the weak equivalences in the fixed point categories. Note that if $X$ is a $G$-space, then the category $R_{hf}(X)$ of homotopy finite retractive spaces over $X$ has a $G$-action. For a retractive space $Y$, $gY$ is defined by precomposing the inclusion map by $g^{-1}$ and postcomposing the retraction map by $g$.  We can apply \autoref{thm:can_deloop_gwalds} to this category, and the resulting theory $\bA_G^\textup{coarse}(X)$ has as its $H$-fixed points the $K$-theory of $H$-\emph{equivariant} spaces over $X$, as we expect, but the weak equivalences are the $H$-maps which are \emph{nonequivariant} homotopy equivalences. Thus, \autoref{thm:can_deloop_gwalds} does not suffice to prove \autoref{thm:agx_exists}.

Although $\bA_G^\textup{coarse}(X)$ does not match our expected input for the $h$-cobordism theorem, it does have a surprising connection to the bivariant $A$-theory of Williams \cite{bruce}:
\begin{thm}[\autoref{prop:coarse_equals_bivariant}]
\label{intro_bivariant} 
There is a natural equivalence of spectra
\[ \bA_G^{\textup{coarse}}(X)^H \simeq \bA(EG \times_H X \ra BH). \]
\end{thm}

In a subsequent paper we will show that under this equivalence, the coassembly map for bivariant $A$-theory agrees up to homotopy with the map from fixed points to homotopy fixed points for $\bA_G^\textup{coarse}(X)$.


In order prove \autoref{thm:agx_exists} it is necessary to modify the weak equivalences in the fixed point categories giving $\bA_G^{\textup{coarse}}(X)^H$, and to do this we use the framework of spectral Mackey functors. These are diagrams over a certain spectral variant of the Burnside category, denoted $\GB$. By celebrated work of Guillou and May, the homotopy theory of $\GB$-diagrams is equivalent to that of genuine $G$-spectra \cite{GM1}. Moreover, there are by now a few different ways to pass from combinatorial, category-theoretic data to diagrams of spectra over $\GB$ \cite{Gmonster, Gmonster2, BO, bohmann_osorno}. In essence, one is allowed to give separately for each $H \leq G$ some permutative category, Waldhausen category, or symmetric monoidal or Waldhausen $\infty$-category $R^H$ whose algebraic $K$-theory will become the $H$-fixed points. The rest of the glue that creates the $G$-spectrum is generated by a large collection of exact functors giving the restrictions, transfers, and sums thereof, between the categories $\{ R^H : H \leq G \}$.

Barwick's approach to managing this large collection of data is to define certain adjoint pairs of functors between the categories $R^H$, satisfying Beck-Chevalley isomorphisms \cite[\S 10]{Gmonster}. These may then be ``unfurled'' to create suitably coherent actions of spans on the categories $R^H$, giving a spectral Mackey functor on the $K$-theory  spectra $K(R^H)$. In \S\autoref{transfer_section}, we describe concretely how spans act on the categories $\{\sC^{hH} : H\leq G\}$ -- this is essentially the application of Barwick's ``unfurling" construction found in \cite[\S 8]{Gmonster2}, but formulated for ordinary Waldhausen categories with a $G$-action. Our variant of this construction is then a  ``Mackey functor of Waldhausen categories" in the sense of Bohmann and Osorno  \cite{bohmann_osorno}, which combined with the theorem of Guillou and May \cite{GM1} gives a genuine $G$-spectrum. This in particular allows an alternative ``spectral Mackey functor" definition of $\bA_G^{\textup{coarse}}(X)$ when one plugs in the category $R(X)$ with the $G$-action described above.

However, as we pointed out, the categories $R_{hf}^H(X)$ are \emph{not} of the form $\sC^{hH}$ -- they have the same objects and maps as $R(X)^{hH}$ but more restricted weak equivalences. In  order to get the desired tom Dieck style splittings of the fixed points, in \S\autoref{sec:final}, we   descend the action of spans on the categories $R(X)^{hH}$ to get a ``Mackey functor of Waldhausen categories" with values $G/H\mapsto  R_{hf}^H(X)$, thereby proving \autoref{thm:agx_exists}. Though we work in the framework of \cite{GM1} and \cite{bohmann_osorno} to build $A_G(X)$, the same constructions appear to also make $R_{hf}^H(X)$ into a Mackey functor of Waldhausen categories within Barwick's framework.

\begin{rem}
There is a ``Cartan'' map
\[ \bA_G(X) \ra \bA_G^\textup{coarse}(X). \]
This becomes a map of genuine $G$-spectra if we define $\bA_G^\textup{coarse}(X)$ using the Mackey structure on $K(\sC^{hH})$.
We believe that this Mackey structure gives the same $G$-spectrum as the one produced by delooping the space $K_G(\sC)$ using \autoref{thm:can_deloop_gwalds}, and that more generally the $K$-theory of group actions from \cite[\S 8]{Gmonster2} gives the same $G$-spectrum as \autoref{thm:can_deloop_gwalds}. The argument we have in mind for the former claim depends on multifunctoriality properties of equivariant $K$-theory that have not yet been carefully established.
\end{rem}

Our constructions are inspired by, but distinct from, the construction of Real algebraic $K$-theory by Hesselholt and Madsen \cite{hesselholt_madsen_real}. We consider Waldhausen categories with (covariant) actions by $G$ through exact functors, whereas the basic input for Real $K$-theory is categories with a contravariant involution. We do not formulate an equivariant version of $\mathcal S_{\sbt}$ here, but we consider this to be a problem of significant importance for future work.

\subsection*{Acknowledgements}
It is a pleasure to acknowledge the contributions to this project arising from conversations with Clark Barwick, Andrew Blumberg, Anna Marie Bohmann, Wojciech Dorabia\l{}a, Emanuele Dotto, Bert Guillou, Tom Goodwillie, Mike Hill, Wolfgang L\" uck, Akhil Mathew, Peter May, Randy McCarthy, Ang{\'e}lica Osorno, John Rognes, Daniel Sch\"{a}ppi, and Stefan Schwede. We particularly thank Clark Barwick for helpful explanations of the $K$-theory of group actions; Anna Marie Bohmann and Ang{\'e}lica Osorno for sharing some of their work in progress that we use in this paper; and John Rognes for sharing with us some of his past unpublished work on equivariant $A$-theory. We are indebted to an anonymous referee for a very careful read and valuable feedback that vastly improved the paper. Finally, we thank the Hausdorff Institute for Mathematics and the Max Planck Institute in Bonn for their hospitality. Both authors were partially supported by AMS Simons Travel Grants. The second named author also acknowledges support from NSF grant DMS 1709461/1850644.

\section{Equivariant $K$-theory of Waldhausen $G$-categories}

Let $G$ be a finite group. In this first section we recall from \cite[\S 2]{thesispaper} the construction $\Cat(\tG,-)$, and how it rectifies pseudo equivariant functors into equivariant ones. In subsection \autoref{strictification} we expand this to a more general strictification result: we show that there is a strictification 2-functor from $G$-categories, pseudo equivariant functors and pseudo equivariant natural transformations to  $G$-categories, equivariant functors and equivariant natural transformations. In subsections \autoref{rect_symm} and \autoref{waldhausen gcat} we give applications of this strictification result to rectifying symmetric monoidal and Waldhausen categories with $G$-action. In subsections \autoref{deloopingsymm} and \autoref{sec:deloop_wald} we show when and how one can deloop symmetric monoidal and Waldhausen categories with $G$-action. In particular, we prove \autoref{thm:can_deloop_gwalds}.

\subsection{Strictification of pseudo equivariance}\label{strictification} Let $G\Cat$ be the 2-category with $0$-cells given by $G$-categories, $1$-cells given by equivariant functors, and $2$-cells given by equivariant natural transformations. For $G$-categories $\sA$ and $\sB$, we define $\Cat(\sA,\sB)$ to be the category of all functors and natural transformations, with $G$ acting by conjugation. More precisely, for $F\colon \sA\to \sB$, 
$g\in G$, and $A$ either an object or a morphism of $\sA$, 
$(gF)(A) = gF(g^{-1}A)$.  Similarly, for a natural transformation 
$\eta\colon E\to F$ and an object $A$ of 
$\sA$, 
$$(g\eta)_A = g\eta_{g^{-1}A}\colon g E(g^{-1}A)\to gF(g^{-1}A).$$
Therefore the fixed point category $\Cat(\sA, \sB)^G$ is the category of equivariant functors and equivariant natural transformations.

\begin{df}
Define $\tG$ to be the  $G$-groupoid with objects the elements of   $G$ and a unique morphism between any two objects. The action of $G$ on the object set $G$ of $\tG$ is by left translation, and this extends in a unique way to an action on the morphisms. Up to $G$-isomorphism, $\tG$ is the translation category of $G$, and its classifying space is the space $EG$. 
\end{df}

\begin{df}
Define  the {\bf homotopy $G$-fixed points}, $\sC^{hG}$, of a $G$-category $\sC$ as $\Cat(\tG, \sC)^G$. \end{df}


\begin{rem}\label{hfixed} For each $H \leq G$, the natural map $\tH\rightarrow \tG$ induced by the inclusion is an equivalence of $H$-categories, meaning it has an $H$-equivariant inverse, and $H$-equivariant natural isomorphisms between both composites and the identity. 
	So, we can unambiguously up to equivalence define the homotopy $H$-fixed points $\sC^{hH}$ as $\Cat(\tG, \sC)^H\simeq \Cat(\ti H, \sC)^H.$
\end{rem}

Recall from \cite[Prop. 2.12]{thesispaper} the following explicit description of the homotopy fixed point category $\Cat(\tG, \sC)^G$. Its objects are objects of $C$ together with isomorphisms $\psi_g\colon C\xrightarrow{\cong} gC$ for all $g\in G$,  such that $\psi_e=\mathrm{id}_C$ and and the following cocycle  condition is satisfied:
\begin{equation}\label{wrong cocycle}
 \psi_{gh}=(g\psi_h)\psi_g.
\end{equation}

A morphism is given by a morphism  $\alpha\colon C\rightarrow  C'$ in $\sC$ such that the following diagram commutes for any $g\in G$ :
$$ \xymatrix{
C \ar[d]_\alpha \ar[rr]^{\psi_g} && g C \ar[d]^{g \alpha}\\
C' \ar[rr]^{\psi'_g} && g C'.} $$

\noindent We recall the following definition \cite[Def. 3.1.]{thesispaper}.
\begin{df}\label{weakequivariant}
A $\emph{pseudo equivariant}$ functor between $G$-categories $\sC$ and $\sD$ is a functor $\Theta \colon \sC \ra \sD$, together with
 natural isomorphisms of functors  $\theta_g$ for all $g\in G$
 
 \[
\xymatrix{
\ \sC\  \ar[r]^-{g\cdot} \ar[d]_{\Theta} \drtwocell<\omit>{<0>\quad   \theta_g}  & \ \sC\  \ar[d]^{\Theta}  \\ 
\ \sD\ \ar[r]_-{g\cdot} & \ \sD\ . \\}
\] 
 such that  $\theta_e=\id$ and  for $g,h \in G$ we have an equality of natural transformations, where on the left hand side we are considering the composite of natural transformations:
\[
\xymatrix{
\ \sC\  \ar[r]^-{h\cdot} \ar[d]_{\Theta} \drtwocell<\omit>{<0>\quad   \theta_h}  & \ \sC\  \ar[d]^{\Theta} \  \ar[r]^-{g\cdot}  \drtwocell<\omit>{<0>\quad   \theta_g}  & \ \sC\  \ar[d]^{\Theta}  \\ 
\ \sD\ \ar[r]_-{h\cdot} & \ \sD\ \ar[r]_-{g\cdot} & \ \sD\ } 
\qquad = \qquad
\xymatrix{
\ \sC\  \ar[r]^-{gh\cdot} \ar[d]_{\Theta} \drtwocell<\omit>{<0>\quad   \theta_{gh}}  & \ \sC\  \ar[d]^{\Theta}  \\ 
\ \sD\ \ar[r]_-{gh\cdot} & \ \sD \ .  }
\]
Requiring this equality makes sense because the outer right down and down right composites in the two diagrams are equal. Explicitly, for $C$ an object of $\sC$, this means that the following diagram commutes:

\[\xymatrix{
\Theta(ghC) \ar[rr]^{\theta_g (hC)} \ar@/^3pc/[rrrr]^{\theta_{gh}(C) \ \ }  & & g\Theta(hC) \ar[rr]^{g \theta_h(C)}  & & gh \Theta(C).
} \]

\end{df}

\begin{rem}\label{afterthetadiagram}
If $\theta_g$ are equalities for all $g\in G$, then $\Theta$ is actually an equivariant functor.
\end{rem}

We may think of a $G$-category as a functor $\mc BG\to\Cat$, where $\mc BG$ is the groupoid with one object and morphism group $G$. Then an equivariant functor is just a natural transformation between the corresponding functors $\mc BG\to \Cat$, and a \emph{pseudo equivariant} functor is a normal pseudo natural transformation. 
Note that the composition of pseudo  equivariant functors $\Phi\circ \Theta$ is again a pseudo equivariant functor with coherence isomorphisms given by $(\phi_g * \Theta) \circ (\Phi * \theta_g)$ where $*$ denotes whiskering. In other words, at an object $C$,  this is the composite
\[ \xymatrix @C=4em{
	\Phi \Theta gC \ar[r]^-{\Phi(\theta_g(C))} & \Phi g \Theta C \ar[r]^-{\phi_g(\Theta C)} & g\Phi \Theta C.
} \]

We check that this satisfies the required cocycle condition:

\[
\xymatrix @C=4em{
\Phi\Theta ghC\ar[rr]^-{\Phi(\theta_g(hC))} \ar[drr]_-{\Phi(\theta_{gh}(C))} && \Phi g\Theta hC \ar[d]^-{\Phi g(\theta_h(C))} \ar[rr]^-{\phi_g(\Theta hC)} && g\Phi\Theta hC \ar[d]^-{g\Phi(\theta_h(C))} \\
&& \Phi gh\Theta C \ar[rr]^-{\phi_g(h\Theta C)}  \ar[drr]_-{\phi_{gh}(\Theta C)} && g\Phi h \Theta C \ar[d]^-{g(\phi_h(\Theta C))} \\
&&&&gh \Phi\Theta C .
}
\]

We give a definition of pseudo equivariant natural transformations between pseudo equivariant functors.
\begin{df}
Let $\eta\colon \Theta \Rightarrow \Psi$ be a natural transformation between pseudo equivariant functors $\sC\to \sD$. We say that $\eta$ is \emph{pseudo equivariant} if the following diagram commutes:

\[\xymatrix{
\Theta(gC)\ar[rr]^{\eta_{gC}} \ar[d]_-{\theta_g(C)}^-\cong &&\Psi(gC) \ar[d]^-{\psi_g(C)}_-\cong\\
g\Theta(C) \ar[rr]_-{g\eta_C} && g\Psi(C).
}\]
\end{df}
\noindent In particular, if $\Theta$ and $\Psi$ are equivariant functors, i.e., if $\theta_g$ and $\psi_g$ are identities, then $\eta$ is an equivariant natural transformation. Note that the composite of two pseudo equivariant natural transformations is also a pseudo equivariant natural transformation.

\begin{df}
We define the 2-category $G\Cat^{\mathrm{pseudo}}$ with 
\begin{itemize}
\item 0-cells given by $G$-categories;
\item 1-cells given by pseudo equivariant functors;
\item 2-cells given by pseudo equivariant natural transformations.
\end{itemize}
\end{df}

\begin{thm}\label{strictificationthm}
The assignment $\sC\mapsto \Cat(\tG, \sC)$ on $0$-cells extends to a 2-functor $G\Cat^{\mathrm{pseudo}}\to G\Cat$.
\end{thm}

We spend the rest of this section giving the necessary constructions of $1$-cells and $2$-cells and proving this theorem. We recall the following construction and result from \cite{thesispaper}, which gives the construction of $1$-cells of the strictification functor. Given a pseudo equivariant functor $\Theta \colon \sC \to \sD$, we construct a functor $$\widetilde{\Theta}\colon \Cat(\tG, \sC)\to \Cat(\tG, \sD),$$ as follows: for a functor $F\colon \tG\to \sC$,  the functor $\widetilde{\Theta}(F)\colon \tG\to \sD$ is defined on objects by
$$\widetilde{\Theta}(F)(g)=g\Theta ((g^{-1}F)(e))=g \Theta(g^{-1} F(g)),$$ and on morphisms $g\rightarrow g'$ it is defined  as the composite
$$ g \Theta(g^{-1} F(g)) \xrightarrow[\cong]{\theta_g^{-1}(g^{-1}F(g))}  \Theta(gg^{-1} F(g))  \xrightarrow{\Theta(F(g\rightarrow g'))}  \Theta(g'g'^{-1} F(g'))  \xrightarrow[\cong]{\theta_{g'}(g'^{-1}F(g'))} g' \Theta(g'^{-1} F(g')).$$

For a morphism in $\Cat(\tG, \sC)$, namely a natural transformation $\alpha\colon F\Rightarrow E$, the components of $\widetilde{\Theta}(\alpha)$ are defined as
$$\widetilde{\Theta}(\alpha)_g=g\Theta(g^{-1}\alpha_g).$$ It was checked in \cite[Prop. 3.3.]{thesispaper} that this is indeed a natural transformation.

\begin{prop}(\cite[Prop. 3.3.]{thesispaper})\label{pseudo equiv}
For a pseudo equivariant functor $\Theta \colon \sC \to \sD$ the  induced functor $$\widetilde{\Theta}\colon \Cat(\tG, \sC)\to \Cat(\tG, \sD),$$ as defined above, is on the nose equivariant. \end{prop}

\begin{rem}
Note that if $\Theta\colon \sC \to \sD$ is equivariant and not only pseudoequivariant, then $\widetilde{\Theta}\colon \Cat(\tG, \sC)\to \Cat(\tG, \sD)$ is the functor induced by postcomposition. If $\Theta$ is not equivariant but only pseudoequivariant, then the functor induced by postcomposition would  not be  equivariant.
\end{rem}

The induced map on homotopy fixed points $\widetilde{\Theta}^H\colon \sC^{hH}\to \sD^{hH}$  
takes an object $C$ with choices of isomorphisms $\psi_g\colon C\xrightarrow{\cong} gC$ to $\Theta(C)$ with isomorphisms $\Theta(C)\xrightarrow{\cong} g\Theta(C)$ defined as the composites $$\Theta(C)\xrightarrow[\cong]{\Theta(\psi_g)} \Theta(gC)\xrightarrow[\cong]{\theta_g(C)} g\Theta(C).$$ It was checked explicitly in \cite{thesispaper} that these composites satisfy the required cocycle condition.

We expand on \autoref{pseudo equiv} to $2$-cells. Suppose $\eta\colon \Theta \Rightarrow \Psi$ is a pseudo equivariant natural transformation between pseudo equivariant functors $\sC\to \sD$. 
%
We define a natural transformation $\widetilde{\eta} \colon\widetilde{\Theta} \Rightarrow \widetilde{\Psi}$ of functors $\Cat(\tG, \sC) \to \Cat(\tG, \sD)$. Let $F$ be an object in $\Cat(\tG, \sC)$. Define $\widetilde{\eta}_F\colon \widetilde{\Theta}\to \widetilde{\Psi}$ to be the natural transformation of functors $\tG\to \sD$ with $g$ component defined by
$$(\widetilde{\eta}_F)_g\colon  g\Theta(g^{-1}F(g))\xrightarrow{g\eta_{g^{-1}F(g)}} g\Psi(g^{-1}F(g)).$$ We note that this gives indeed a natural transformation, since for a map $g\to h$ in $\tG$, the following naturality diagram commutes---the upper and lower squares commute because we assumed $\eta$ is pseudoequivariant and the middle square is the naturality square for $\eta$:
\[\xymatrix{
g\Theta(g^{-1}F(g))\ar[d]_-{\theta_g^{-1}(g^{-1}F(g))} \ar[rr]^-{g\eta_{g^{-1}F(g)}} && g\Psi(g^{-1}F(g)) \ar[d]^-{\psi^{-1}_g(g^{-1}F(g))}\\
\Theta(F(g)) \ar[d] \ar[rr]^-{\eta_{F(g)}}&& \Psi(F(g))\ar[d]\\
\Theta(F(h)) \ar[d]_-{\theta_h(h^{-1}F(h))} \ar[rr]_-{\eta_{F(h)}} &&\Psi(F(h))\ar[d]^-{\psi_h(h^{-1}F(h))}\\
h\Theta(h^{-1}F(h)) \ar[rr]_-{g\eta_{h^{-1}F(h)}} && h\Psi(h^{-1}F(h)).
}\]

Similarly, for a natural transformation $\alpha\colon F\Rightarrow E$ of functors $\tG\to \sC$, the necessary naturality diagram of natural transformations of functors $\tG\to \sD$
\[\xymatrix{
\widetilde{\Theta}(F) \ar[d]_-{\widetilde{\Theta}(\alpha)}\ar[rr]^-{\widetilde{\eta}_F}&& \widetilde{\Psi}(F)\ar[d]^-{\widetilde{\Psi}(\alpha)}\\
\widetilde{\Theta}(E) \ar[rr]_-{\widetilde{\eta}_E} && \widetilde{\Psi}(E)
}\]
translates, on component $g$, to the following diagram

\[\xymatrix @R=2em @C=4.5em{
g\Theta(g^{-1}F(g)) \ar[d]_-{g\Theta(g^{-1}\alpha_g)} \ar[r]^-{\theta_g^{-1}(g^{-1}F(g))} & \Theta(F(g))\ar[d]_-{\Theta(\alpha_g)} \ar[r]^-{\eta_{F(g)}}& \Psi(F(g)) \ar[d]^-{\Psi(\alpha_g)} \ar[r]^-{\psi_g(g^{-1}F(g))} & g\Psi(g^{-1}F(g)) \ar[d]^-{g\Psi(g^{-1}\alpha_g)}\\
g\Theta(g^{-1}E(g)) \ar[r]_-{\theta_g^{-1}(g^{-1}E(g))} & \Theta(E(g))  \ar[r]_-{\eta_{E(g)}}& \Psi(E(g)) \ar[r]_-{\psi_g(g^{-1}E(g))} &
g\Psi(g^{-1}E(g)).
}\]
The outer squares commute by the naturality of the isomorphisms $\theta_g$ and $\psi_g$, and the middle square commutes by the naturality of $\eta$.


\begin{prop}
For a pseudo equivariant natural transformation $\eta\colon \Theta \Rightarrow \Psi$  between pseudo equivariant functors $\sC\to \sD$, the natural transformation $\widetilde{\eta} \colon\widetilde{\Theta} \Rightarrow \widetilde{\Psi}$, as defined above, is an equivariant natural transformation between equivariant functors.
\end{prop}

\begin{proof} We can see that the natural transformation  $\widetilde{\eta} \colon\widetilde{\Theta} \Rightarrow \widetilde{\Psi}$ is on the nose equivariant: we check that for any $h\in G$ and $F\in \Cat(\tG, \sC)$, we have an equality $h\widetilde{\eta}_F = \widetilde{\eta}_{hF}$. Note that the $g$ component
$$(h\widetilde{\eta}_F)_g\colon h(\widetilde{\Theta}(F))(g) \xrightarrow{} h(\widetilde{\Psi}(F))(g)$$ is equal to
$$hh^{-1}g\Theta((h^{-1}g)^{-1}F(h^{-1}g))  \xrightarrow{hh^{-1}g\eta_{(h^{-1}g)^{-1}F(h^{-1}g)}} hh^{-1}g\Psi((h^{-1}g)^{-1}F(h^{-1}g)).$$
On the other hand, the $g$ component
$$ (\widetilde{\eta}_{hF})_g\colon g\Theta(g^{-1}(hF)(g)) \to g\Psi (g^{-1}(hF)(g)) $$ is equal to
$$g\Theta(g^{-1}hF(h^{-1}g)) \xrightarrow{g\eta_{g^{-1}(hF)(g)}} g\Psi(g^{-1}hF(h^{-1}g)).$$
Thus $(h\widetilde{\eta}_F)_g= (\widetilde{\eta}_{hF})_g$. 
\end{proof}

In order to show that we have defined a 2-functor $G\Cat^{\mathrm{pseudo}}\to G\Cat$ and thus complete  the proof of \autoref{strictificationthm}, we need to show that composition of functors, identity functors, identity natural transformations, and both horizontal and vertical composition of natural transformations is strictly preserved. We leave the straightforward check that $\widetilde{\Theta\circ\Psi} =\widetilde{\Theta}\circ \widetilde{\Psi}$, $\widetilde{\eta_1 * \eta_2}=\widetilde{\eta_1} * \widetilde{\eta_2}$, $\widetilde{\eta_1\circ \eta_2}=\widetilde{\eta_1}\circ\widetilde{\eta_2}$, and that $\widetilde{\id}=\id$ both on functors and natural transformations to the reader.


\subsection{Rectification of symmetric monoidal $G$-categories}\label{rect_symm}
One application of \autoref{strictificationthm} is to strictify $G$-actions on symmetric monoidal categories.
Suppose $\sC$ is a symmetric monoidal category, with a $G$-action that preserves the symmetric monoidal structure $\oplus$ up to coherent isomorphism. In other words, $\sC$ is a functor $\mc BG\rightarrow \mathrm{Sym} \Cat^\mathrm{strong}$ from $\mc BG$ to the category of strict symmetric monoidal categories and \emph{strong} monoidal functors. Then the symmetric monoidal structure map $$\sC\times \sC \xrightarrow{\oplus} \sC$$ is pseudoequivariant, where the $G$-action on $\sC\times \sC$ is diagonal. In addition, we get coherent isomorphisms $gI\cong I$ for every $g\in G$, where $I$ is the unit object of $\sC$.

If $\sC$ is such a symmetric monoidal category, then $\Cat(\tG, \sC)$ is a symmetric monoidal category whose sum and unit are strictly $G$-equivariant. This is because \autoref{pseudo equiv} gives an on the nose equivariant functor
$$\oplus\colon \Cat(\tG, \sC\times \sC) \cong \Cat(\tG, \sC) \times \Cat(\tG, \sC) \ra \Cat(\tG, \sC)$$
which we take as the sum in $\Cat(\tG, \sC)$. The unit is the functor $F_I\colon \tG \ra \sC$ defined by $F_I(g)=gI$, where $I$ is the unit of $\sC$. Explicitly, $F_1\oplus F_2$
in $\Cat(\tG, \sC)$  is defined on objects as
$$(F_1\oplus F_2)(g)= g\big(g^{-1}F_1(g)\oplus g^{-1}F_2(g)\big),$$ which, of course, is the same as $F_1(g)\oplus F_2(g)$ when the $G$-action on $\sC$ preserves $\oplus$ strictly, and a morphism $g\rightarrow g'$, it is defined as 
$$(F_1\oplus F_2) (g)\xrightarrow{\cong} F_1(g)\oplus F_2(g) \xrightarrow{F_1(g\rightarrow g')\oplus F_2(g\rightarrow g')}    F_1(g')\oplus F_2(g')\xrightarrow{\cong} (F_1\oplus F_2)(g').$$
These sum formulas motivate our definition of transfers on $\mc C^{hH}$ in \S\autoref{transfer_section} below.

When we take the $K$-theory of $\sC$ below, we will actually want to strictify $\sC$ in two ways: we will want to make the $G$-action commute with the sum strictly, but we will also want to strictify the symmetric monoidal category $\sC$ to a monoidally $G$-equivalent permutative category with $G$-action. We give the details in the discussion before \autoref{rectify_symmetric}.

\subsection{Rectification of Waldhausen $G$-categories}\label{waldhausen gcat}
Now suppose that $\sC$ is a Waldhausen category with $G$-action through exact functors. 
In other words, for each $g\in G$ the functor $g\cdot \colon \sC \rightarrow \sC$ preserves cofibrations, weak equivalences, the zero object, and pushouts along cofibrations. (In fact, the last two are automatic since $g\cdot$ is an isomorphism of categories.) However, we emphasize that $g\cdot$ preserves the zero object and pushouts only \emph{up to unique isomorphism,} and not on the nose.

\begin{obs}\label{waldfixedpts}
In general, the fixed point category $\sC^H$ is \emph{not} a Waldhausen category, because it is not closed under pushouts. A pushout diagram in $\sC^H$ has a pushout in $\sC$, but it is only preserved by the $H$-action up to isomorphism, and so in general it does not lie in $\sC^H$. 
\end{obs}

We will get around this by showing that the homotopy fixed points $\sC^{hH}$ form a Waldhausen category (\autoref{htpyfixedpts}). First we check that $\Cat(\tG, \sC)$ is a Waldhausen $G$-category, by defining the cofibrations and weak equivalences pointwise. More precisely, for $F_1, F_2 \in \Cat(\tG, \sC)$ , $$F_1 \xrightarrow{\eta} F_2 $$ is a cofibration or a weak equivalence if for every $g\in \tG$, the map 
$F_1(g) \rightarrow F_2(g) $ is a cofibration or a weak equivalence, respectively, in $\sC$. If we define the zero object and pushouts in a pointwise manner, they will not be fixed, so we show a little more care:

\begin{lem}\label{zero}
There is a zero object in $\Cat(\tG, \sC)$, which is $G$-fixed.
\end{lem}

\begin{proof}
Consider the functor $Z\colon\ast\ra\sC$ from the one object category $\ast$ to $\sC$, which picks out the zero object 0 of  $\sC$. Note that this functor is not equivariant since $ 0\neq g\cdot 0$, but for every $g$ we have a unique isomorphism $\theta_g\colon0\xrightarrow{\cong} g\cdot 0$. Since these isomorphisms are unique, it must be that the isomorphisms $0\xra{\theta_g} g\cdot 0 \xra{g \theta_h} (gh)\cdot 0$ and $0\xra{\theta_{gh}} (gh)\cdot 0$ coincide, and therefore $Z$ is pseudo equivariant. 

By \autoref{pseudo equiv}, since $Z$ is pseudo equivariant, there is an induced on the nose equivariant functor $\ast\cong \cat(\tG, \ast) \ra \cat(\tG, \sC),$ which sends the one object of $\ast$ to the functor $F_0\in \Cat(\tG, \sC)$ defined on objects by  $F_0(g)=g\cdot0$, and defined on the unique morphism  
from $g$ to $h$   by composing the unique isomorphisms $0\cong g\cdot 0$ and $0\cong h\cdot 0$ to get an isomorphism $g\cdot 0 \xra{\cong} h\cdot 0$ in $\sC$. Since the functor $\ast\ra\cat(\tG, \sC)$ with value $F_0$ is equivariant by \autoref{pseudo equiv},  the object $F_0$ of $\cat(\tG, \sC)$ lies in the $G$-fixed point subcategory. It is easy to check that this is a zero object in $\Cat(\tG, \sC)$.

\end{proof}

\begin{lem}\label{pushout}
There exist pushouts along cofibrations in $\Cat(\tG, \sC)$, so that pushouts of $H$-fixed diagrams are $H$-fixed. 
\end{lem}

\begin{proof}
The same argument as in the previous proof applies: if one considers the category of $\mc P(\mc C)$ pushout diagrams along cofibrations, and a functor $\mc P(\mc C) \ra \mc C$ which assigns to each pushout diagram along a cofibration

\[
\xymatrix{
A \ar[d] \ \ar@{>->}[r] & B \\
C &
}
\]

 \noindent a choice $P$ of pushout $B\amalg_A C$, this functor is not equivariant. However, the canonical isomorphisms $g\cdot (B\amalg_A C) \cong (g\cdot B)\coprod_{(g\cdot A)} (g\cdot  C)$ that exist for any pushout in $\sC$ and any $g\in G$ since we are assuming $g\cdot$ is an exact functor, ensure that the functor $\mc P(\mc C) \ra \mc C$ is pseudo equivariant. Therefore by \autoref{pseudo equiv}, we get the nose equivariant functor
\[ \mc P(\Cat(\tG,\mc C)) \cong \Cat(\tG,\mc P(\mc C)) \ra \Cat(\tG,\mc C) \]
Since pushouts of functors are defined objectwise, this assigns to each diagram in $\Cat(\tG,\mc C)$ a pushout, and if the diagram is $H$-fixed then the pushout is $H$-fixed as well.
\end{proof}
From the construction of the corresponding equivariant functor from a pseudo equivariant functor in the proof of \autoref{pseudo equiv}, we get  an explicit description for the pushouts in $\Cat(\tG, \sC)$. For a diagram
\[
\xymatrix{
F_1 \ar[d] \ \ar@{>->}[r] & F_2 \\
F_3 &
}
\]
in $\Cat(\tG, \sC)$, the pushout ${\bf P}\colon \tG\to \sC$ is defined on objects by
$${\bf P}(g)= g\cdot (g^{-1} F_3(g) \coprod_{g^{-1}F_1(g)} g^{-1} F_2(g) ).$$

If the pushout diagram is $G$-fixed, then the pushout {\bf P} is defined by ${ \bf P}(e)=P$, where $P$ is a pushout of the above diagram evaluated at $e$, and ${\bf P}(g)= g\cdot P.$ On morphisms, ${\bf P}(g, g')$ is the composite of the unique isomorphisms $P\cong g\cdot P$ and $P\cong g'\cdot P$.

\begin{thm}\label{htpyfixedpts}
Let $\sC$ be a $G$-equivariant Waldhausen category, and let $H$ be a subgroup of $G$. Then $\sC^{hH}$ is a Waldhausen category with cofibrations and weak equivalences the $H$-fixed cofibrations and weak equivalences in $\Cat(\tG, \sC)$. \end{thm}

\begin{proof}

Note that composition of $H$-fixed maps is $H$-fixed, thus the classes of cofibrations and weak equivalences in $\Cat(\tG, \sC)^H$ are closed under composition, and an $H$-fixed isomorphism is in particular a $H$-fixed cofibration and weak equivalence. 

By \autoref{zero}, there is a zero object $F_0$ in $\Cat(\tG, \sC)^H$. Moreover,  for any functor $F$ in $\Cat(\tG, \sC)$, each map $F_0(g) \rightarrowtail F(g)$ is a cofibration since it is the composite of $g\cdot 0\cong 0 $ and the unique map $ 0 \rightarrowtail F(g)$, which are both cofibrations. Thus the map $F_0\rightarrowtail F$ is by definition a cofibration. 

By \autoref{pushout}, for a pushout diagram along a cofibration in $\Cat(\tG, \sC)^H$, there exists a pushout in this fixed point subcategory. The gluing axiom for weak equivalences is inherited from $\sC$.
\end{proof}

Note that the equivalence from \autoref{hfixed}
$\Cat(\tG, \sC)^H\simeq \Cat(\ti H, \sC)^H$ is an equivalence of Waldhausen categories. 

\subsection{Delooping symmetric monoidal $G$-categories}\label{deloopingsymm}
Classical operadic infinite loop space theory \cite{MayGeo} gives a machine for constructing, from  a space $X$ with an action by an $E_\infty$ operad, an $\Omega$-spectrum whose zeroth space is the group completion of $X$. If in addition $X$ has an action of a finite group $G$ through $E_\infty$ maps, then the resulting spectrum has a $G$-action, namely it is a \emph{na\"ive} $\Omega$-$G$-spectrum. By definition, ``na\"ive" means that the deloopings are only for spheres with trivial $G$-action.

In order to get deloopings by representation spheres $S^V$ for all finite-dimensional representations $V$ of $G$, the $G$-space $X$ needs to be an algebra over a genuine $E_\infty$-$G$-operad.
The difference between a na\"ive and a genuine $E_\infty$ operad $\sO$ lies in the fixed points of the $G \times \Sigma_n$-space $\sO(n)$ for each $n$. For each subgroup $\Lambda \leq G \times \Sigma_n$ we have:
\begin{center}
\begin{tabular}{c|c|c|c}
& $(\Lambda \cap \Sigma_n) \neq \{1\}$ & $(\Lambda \cap \Sigma_n) = \{1\}$ & $
(\Lambda \cap \Sigma_n) = \{1\}$ \\
&& $\Lambda \cap G=\{1\}$ &  $\Lambda \cap G \neq \{1\}$  \\\hline
na\"ive $E_\infty$ operad & $\sO(n)^\Lambda = \emptyset$ & $\sO(n)^\Lambda = \emptyset$ & $\sO(n)^\Lambda \simeq *$ \\
genuine $E_\infty$ operad & $\sO(n)^\Lambda = \emptyset$ & $\sO(n)^\Lambda \simeq *$ & $\sO(n)^\Lambda \simeq *$
\end{tabular}
\end{center}

\begin{rem} In a na\"ive $E_\infty$ operad $\sO$, the spaces $\sO(n)$ are the total spaces of universal principal $\Sigma_n$-bundles with $G$-action, whereas in a genuine $E_\infty$ operad $\sO_G$, the spaces $\sO_G(n)$ are the total spaces of equivariant universal principal $G$-$\Sigma_n$-bundles. For a thorough discussion of equivariant bundle theory, see \cite[Ch.VII]{alaska}.
\end{rem}

There are a few different machines that produce these equivariant deloopings, though they are all equivalent \cite{MMO}. We will focus on the machine of Guillou and May \cite{GM3}. Consider the categorical Barratt-Eccles operad $\sO(j) = \ti \Sigma_j$, and apply $\Cat(\tG,-)$ levelwise. Since $\Cat(\tG, -)$ preserves products, this gives an operad $$\sO_G(j)=\Cat(\tG, \ti \Sigma_j)$$ in $G$-categories. Guillou and May show that the levelwise realizations $|\sO_G(j)|$ then form a genuine $E_\infty$-operad in unbased $G$-spaces.
\begin{thm}[\cite{GM3}]\label{omega}
There is a functor $\bK_G(-)$ from $|\sO_G|$-algebras $X$ to orthogonal $G$-spectra, whose output is an $\Omega$-$G$-spectrum in the sense that the maps
$$\bK_G(X)(V) \to \Omega^{W-V}\bK_G(X)(W)$$
are equivariant equivalences. There is a natural equivariant group completion map $$X\ra \bK_G(X)(0)$$ and a natural weak equivalence of nonequivariant orthogonal spectra
$$\bK(X^H) \to (\bK_G X)^H$$
for all subgroups $H$ of $G$.
\end{thm}
Recall that an equivariant group completion is a map that is a group completion on the $H$-fixed points for all subgroups $H$ of $G$. In particular, if the fixed points $X^H$ are connected for all subgroups $H$, then the map $X\ra \bK_G(X)(0)$ is an equivalence.

Since realization is a symmetric monoidal functor, if $\sC$ is a $G$-category with an action of $\sO_G$, its classifying space $|\sC|$ is an algebra over $|\sO_G|$ in $G$-spaces. We are therefore interested in constructing examples of $\sO_G$-algebras $\mc C$. We first recall that a category $\sC$ with an action of the Barratt-Eccles operad $\sO$ in $\Cat$ is a \emph{permutative category}, i.e. it is symmetric monoidal with strict unit and strict associativity \cite{MayPerm2}. Any symmetric monoidal category $\sC$ can be rectified to an equivalent permutative category by a well known trick of MacLane \cite{MacLane}. The MacLane strictification functor $(-)^\mathrm{str}\colon \mathrm{Sym}\Cat^\mathrm{strong} \to \mathrm{Sym}\Cat^\mathrm{strict}$, from the category of symmetric monoidal categories and strong symmetric monoidal functors to the category of strict symmetric monoidal categories and strict symmetric monoidal functors, is the left adjoint of the forgetful map $\mathbb{U}$. The category $\sC^\mathrm{str}$ has as objects lists  $(c_1,\ldots,c_n)$ of objects in $\sC$ with sum given by concatenation, and morphisms between  $(c_1,\ldots,c_n)$ and  $(d_1,\ldots,d_m)$ are given by morphisms $c_1\oplus \ldots \oplus c_n\to d_1\oplus \ldots d_m$ in $\sC$, where iterated uses of the monoidal product are parenthensized to the left.

If $\sC$ has a coherent $G$-action as in \S\autoref{rect_symm}, then the composition $\mc BG \to \mathrm{Sym} \Cat^\mathrm{strong} \to \mathrm{Sym} \Cat^\mathrm{strict}$ describes $\sC^\mathrm{str}$ as a category with a $G$-action that commutes with the symmetric monoidal product strictly. This action is defined on objects by $g(c_1,\ldots, c_n)=(gc_1,\ldots, gc_n)$, and on morphisms by
$$gc_1\oplus\ldots \oplus gc_n\cong g(c_1\oplus \ldots \oplus c_n)\to g(d_1\oplus \ldots \oplus d_m)\cong (gd_1\oplus \ldots \oplus gd_m).$$
The components of the unit of the adjunction $\eta\colon \mc C \to \mathbb{U} \mc C^\mathrm{str}$ are strong symmetric monoidal equivalences of symmetric monoidal categories with inverses $\eta^{-1}$ sending the list $(c_1,\ldots,c_n)$ to $c_1 \oplus \ldots \oplus c_n$. We have observed in \cite{thesispaper} that the equivalence of $\sC$ and $\sC^\mathrm{str}$ is through $G$-equivariant functors, when the action on $\sC$ commutes with $\oplus$ strictly. However, now we are assuming that $g$ commutes with $\oplus$ only up to coherent isomorphism. In this case, $\eta$ is still equivariant, but the inverse equivalence $\eta^{-1}$ is only pseudo-equivariant. After applying $\Cat(\tG, -)$, we conclude by \autoref{strictificationthm} that $\eta$ and $\eta^{-1}$ give a $G$-equivariant monoidal equivalence of categories
\[ \Cat(\tG,\mc C) \simeq \Cat(\tG,\mc C^\mathrm{str}). \]

We summarize this discussion in the next proposition.

\begin{prop}\label{rectify_symmetric}
Let $\sC$ be a symmetric monoidal category with $G$-action given through strong monoidal endofunctors. Then the symmetric monoidal $G$-category $\Cat(\tG, \sC)$ is $G$-equivalent to the $\sO_G$-algebra $\Cat(\tG, \sC^\mathrm{str})$.
\end{prop}

We may therefore deloop the classifying space $|\Cat(\tG, \sC)|$ by representations, simply by applying \autoref{omega} to the equivalent classifying space $|\Cat(\tG, \sC^\mathrm{str})|$.

\subsection{Delooping Waldhausen $G$-categories}\label{sec:deloop_wald}
Recall that the algebraic $K$-theory space of the Waldhausen category $\mc C$ is defined as $\Omega|w \mc S_{\sbt} \mc C|$, where $\mc S_{\sbt} \mc C$ is the simplicial Waldhausen category constructed in \cite{waldhausen}. The $w$ means that we restrict to the subcategory of weak equivalences when we take the nerves of the categories $w\mc S_n \mc C$ for varying $n$, before taking the realization of the resulting bisimplicial set $w  N_{\sbt} \mc S_{\sbt} \mc C$.

This is an infinite loop space whose deloopings are given by iterations of the $S_{\sbt}$-construction. However Waldhausen remarks that it is enough to apply $S_{\sbt}$ once, which has the effect of splitting the exact sequences, and then to use an alternate infinite loop space machine with the group completion property on the space $|w \mc S_{\sbt} \mc C|$. Waldhausen notes that the comparison can be achieved by fitting the two resulting spectra into a bispectrum, and a detailed  proof of this result is written down in \cite{coassembly}. We will use this idea to produce equivariant deloopings of Waldhausen $G$-categories.

Suppose that $\sC$ is a Waldhausen category, with an action of $G$ through exact functors. We give $\Cat(\tG, \sC)$ the Waldhausen category structure defined in \S\autoref{waldhausen gcat}. The $G$-action on $\Cat(\tG, \sC)$ induces a $G$-action on the simplicial Waldhausen category $\mc S_{\sbt} \Cat(\tG, \sC)$, which commutes with fixed points:
$$(\mc S_{\sbt} \Cat(\tG, \sC))^H\cong \mc S_{\sbt} (\Cat(\tG, \sC)^H).$$
\begin{rem}
It does not make sense to ask whether $\mc S_{\sbt}$ commutes with fixed points in general, because the fixed point categories $\sC^H$ do not in general have Waldhausen structure.
\end{rem}

\begin{df}\label{df:deloop_waldhausen}
We define the algebraic $K$-theory $G$-space of a Waldhausen $G$-category $\sC$ as 
$$ K_G(\sC) := \Omega|w \mc S_{\sbt} Cat(\ti G,\sC)|$$
\end{df}

From the above discussion, the $H$-fixed points of this space coincide with the algebraic $K$-theory space of the Waldhausen category $\mc C^{hH}$.




\begin{thm}\label{inf loop}
The space $K_G(\sC)$ is an infinite loop $G$-space.
\end{thm}

\begin{proof} 
As in the proof of \autoref{pushout}, we make a choice of coproduct for any pair of objects in $\sC$. By forgetting structure, each Waldhausen $G$-category $\sC$ is a symmetric monoidal $G$-category under the coproduct $\vee$. The $G$-coherence is automatic because each $g$ acts by exact endomorphisms of the category, and therefore preserves coproducts up to canonical isomorphism. 

By  \autoref{rectify_symmetric} we obtain an $\sO_G$-algebra $\Cat(\tG, \sC^\mathrm{str})$ that is monoidally $G$-equivalent to $\sC$. Since we have an actual $G$- equivalence of categories between $$\Cat(\tG, \sC) \rightleftarrows \Cat(\tG, \sC^\mathrm{str}),$$
$\Cat(\tG, \sC^\mathrm{str})$ has Waldhausen structure obtained by transporting the Waldhausen structure of $\Cat(\tG, \sC)$ along the equivalence, so that the functors in the equivalence are exact. By applying $\mc S_{\sbt}$, we obtain a simplicial $\sO_G$-algebra $\mc S_{\sbt} \Cat(\tG, \sC^\mathrm{str})$.
By the gluing lemma, a coproduct of weak equivalences is also a weak equivalence, so the subcategories of weak equivalences $w \mc S_{\sbt} Cat(\tG,\sC^\mathrm{str})$ also form a simplicial $\sO_G$-algebra.



Since the nerve and geometric realization functors are symmetric monoidal, the space $|w \mc S_{\sbt} Cat(\ti G,\sC^\mathrm{str})|$ is an $|\sO_G|$-algebra, and we have an equivalence of $G$-spaces $$|w \mc S_{\sbt} Cat(\ti G,\sC)| \simeq |w \mc S_{\sbt} Cat(\ti G,\sC^\mathrm{str})|.$$ Furthermore, since geometric realization and $\mc S_{\sbt}$ commute with taking fixed points of $\Cat(\tG, \sC)$, we get a homeomorphism 
\[ |w \mc S_{\sbt} Cat(\ti G,C)|^H \cong |w \mc S_{\sbt} Cat(\ti G,C)^H|. \]
These spaces are all connected, so the $G$-space $|w \mc S_{\sbt} Cat(\ti G,C)|$ is already group complete in the equivariant sense. By \autoref{omega} it is therefore an infinite loop $G$-space.
\end{proof}

\begin{df}
For a Waldhausen $G$-category $\sC$, define $\bK_G (\sC)$ as the orthogonal $\Omega$-$G$-spectrum with zeroth space $K_G(\sC)$ obtained by looping once the spectrum given by applying  \autoref{omega}. \end{df}

\begin{prop}\label{fixed_points_agree}
For every subgroup $H$ of $G$, the orthogonal fixed point spectrum $\bK_G(\sC)^H$ is equivalent to the prolongation to orthogonal spectra of the Waldhausen $K$-theory symmetric spectrum of $\sC^{hH}$ defined by iterating the $\mc S_{\sbt}$-construction. \end{prop}
\begin{proof}
By \autoref{omega}, we get that $$\bK_G(\sC)^H\simeq \Omega \bK(|w\mc S_{\sbt} \sC^{hH}|), $$ where $\bK$ is the nonequivariant operadic infinite loop space machine landing in orthogonal spectra. By \cite[Thm 3.11.]{coassembly}, the orthogonal spectrum above is equivalent to the prolongation of the symmetric spectrum of $\sC^{hH}$ defined by $\Omega|w\mc S^{(n)}_{\sbt}\sC^{hH}|$, which is Waldhausen's $K$-theory spectrum of the Waldhausen category $\sC^{hH}$.

\end{proof}

\begin{rem}
The argument \cite[Thm 3.11.]{coassembly}  applies verbatim for a Waldhausen category with $G$-action to give an equivalence of na{\"i}ve $G$-spectra. In particular, by applying the argument to the category with $G$-action $\Cat(\tG, \sC)$, we can conclude that the underlying na{\"i}ve orthogonal $G$-spectrum of $\bK_G(\sC)$ is $G$-equivalent to the prolongation of the symmetric spectrum with $G$-action $\Omega |w\mc S^{(n)}_{\sbt} \Cat(\tG, \sC)|$. On fixed points $\sC^{hH}$, the equivalences are obtained by repeating the nonequivariant argument for each $H$, since $\mc S_{\sbt}$ commutes with taking fixed points of $\Cat(\tG, \sC)$.

\end{rem}

\section{The Waldhausen $G$-category of retractive spaces $R(X)$}


Let $G$ be a finite group and let $X$ be an unbased space with a continuous left $G$-action. 
Let $R(X)$ be the category of non-equivariant retractive spaces over $X$. That is, an object of $R(X)$ is an unbased space $Y$ and two maps
\[ X \overset{i_Y}\ra Y \overset{p_Y}\ra X \]
which compose to the identity on $X$. A morphism $f$ in $R(X)$ is given by the following commutative diagram:
$$\xymatrix @R=1em {
& Y \ar[dr]^{p_Y} \ar[dd]^f & \\
X \ar[ur]^{i_Y} \ar[dr]_{i_{Y'}} && X.\\
& Y' \ar[ur]_{p_{Y'}}&
}$$

\subsection{Action of $G$ on $R(X)$}\label{g_action_on_rx}
The category $R(X)$ inherits a left action by $G$, which we describe explicitly. For any $g\in G$, the functor $g\colon R(X)\to R(X)$ sends an object 
\[ X \overset{i_Y}\ra Y \overset{p_Y}\ra X \] to the object
\[X\overset{g^{-1}}\ra X \overset{i_Y}\ra Y \overset{p_Y}\ra X\overset{g}\ra X. \]

\noindent For a map $f\colon (Y, i_Y, p_Y)\to (Y', i_{Y'}, p_{Y'})$, the map $gf$ is defined by the diagram

$$\xymatrix @R=1em {
& Y \ar[dr]^{g \circ p_Y} \ar[dd]^f & \\
X \ar[ur]^{i_Y \circ g^{-1}} \ar[dr]_{i_{Y'} \circ g^{-1}} && X\\
& Y' \ar[ur]_{g \circ p_{Y'}} &
}$$
which clearly also commutes.

We take the weak equivalences in $R(X)$ to be the weak homotopy equivalences, and the cofibrations to be the the maps that have the fiberwise homotopy extension property (FHEP). In \cite{ms}, these are called the $f$-cofibrations. Then the subcategory of cofibrant objects is a Waldhausen category. By abuse of notation, we will also call this subcategory $R(X)$. It is easy to check that the $G$-action we defined above is through exact functors. 

\subsection{Homotopy fixed points of $R(X)$}\label{rhx}
Recall that the fixed point categories $R(X)^H$ may not be Waldhausen (\autoref{waldfixedpts}). In fact, if $X$ has a nontrivial $G$-action, the category $R(X)^G$ is empty and hence fails to contain a zero object.

However by \autoref{htpyfixedpts}, the homotopy fixed point categories $R(X)^{hH}$ have a Waldhausen category structure. In this case, they admit a more explicit description. 

\begin{prop}\label{rxhh}
The Waldhausen category $R(X)^{hH}$ is equivalent to the Waldhausen category with:
\begin{itemize}
\item objects, the $H$-equivariant retractive spaces over $X$, i.e. the space $Y$ has a left action by $H$, the maps $i_Y$ and $p_Y$ are equivariant;
\item morphisms, the $H$-equivariant maps of retractive spaces $Y \to Y'$;
\item cofibrations, the $H$-equivariant maps which are nonequivariantly cofibrations;
\item  weak equivalences, the $H$-equivariant maps which are nonequivariantly weak equivalences.
\end{itemize}
\end{prop}

\begin{proof}
By \autoref{hfixed}, it is enough to prove the result for $H=G$. The objects of the homotopy fixed point category $R(X)^{hG}= \Cat(\tG, R(X))^G$ are retractive spaces $(Y, i_Y, p_Y)$ together with isomorphisms $\psi_g\colon Y \xrightarrow{\cong} Y$ for all $g$ making the following diagram commute:
$$ \xymatrix{
X \ar[rr]^{i_Y} \ar[d]^{g^{-1}} && Y \ar[d]^{\psi_g} \ar[rr]^{p_Y} && X \\
X \ar[rr]^{i_Y} && Y \ar[rr]^{p_Y} && X\ar[u]^g
}$$
We define the left $G$-action on $Y$ by having $g^{-1}$ act by $\psi_g$. The commutativity of the above diagram implies that $i_Y$ and $p_Y$ are equivariant.  It is then clear that the maps in $R(X)^{hG}$ are the $G$-equivariant maps.
\end{proof}

\begin{rem}
Note that the proof of \autoref{rxhh}  actually gives  an isomorphism of categories when $H=G$, but when $H < G$ we only get an equivalence in light of the equivalence of Waldhausen categories  $ \Cat(\tG, R(X))^H\simeq \Cat(\tH, R(X))^H$ from  \autoref{hfixed}.
\end{rem}

Before taking $K$-theory, we will restrict to a subcategory of finite objects. Let $R_{hf}(X)  \subseteq R(X)$ denote the subcategory of retractive spaces that are \emph{homotopy finite}, i.e., a retract in the homotopy category of an actual finite relative cell complex over $X$.

Clearly the action of $G$ on $R(X)$ respects this condition, and so restricts to a $G$-action on $R_{hf}(X)$. The proof of \autoref{rxhh} applies verbatim to give us that $R_{hf}(X)^{hH} = \Cat(\tG,R_{hf}(X))^H$ is the Waldhausen category of retractive $H$-equivariant spaces over $X$ whose underlying space is homotopy finite.

\begin{rem}\label{rmk:could_be_strong}
By Waldhausen's approximation theorem, if we restrict to the subcategory of spaces that are  homotopy equivalent to cell complexes, with the homotopy equivalences on the total space and the HEP cofibrations, we get equivalent $K$-theory.
\end{rem}

\subsection{Definition of $\bA^{\! \textup{coarse}}_G(X)$}

Applying \autoref{df:deloop_waldhausen} and \autoref{inf loop} to the category of retractive spaces $R_{hf}(X)$ provides our first equivariant generalization of Waldhausen's functor.
\begin{df} We define the $G$-space $A_G^{\! \textup{coarse}}(X) := \Omega | w \mc S_{\sbt} \Cat(\tG, R_{hf}(X)) |$. \end{df}
\begin{cor}\label{coarseomega} The $G$-space $A_G^{\! \textup{coarse}}(X)$ is the zeroth space of a $\Omega$-$G$-spectrum $\bA_G^{\! \textup{coarse}}(X)$. \end{cor}

The upper script ``coarse" indicates that the $H$-fixed point spectrum is the nonequivariant $K$-theory of the category of $H$-equivariant retractive spaces over $X$ with the coarse equivalences. We will proceed to explain how this fixed point spectrum is related to Williams's bivariant $A$-theory functor $\bA(E \to B)$.

\subsection{Relation to bivariant $A$-theory}
For each fibration $p\colon E \to B$ into a cell complex $B$, form a Waldhausen category $R_{hf}(E \overset{p}\to B)$ whose objects are retractive spaces
\[ E \overset{i_Y}\ra Y \overset{p_Y}\ra E \overset{p}\ra B \]
for which $p \circ p_Y$ is a fibration, and over each point $b \in B$ the retractive space $Y_b$ over the fiber $E_b$ is homotopy finite. The weak equivalences are the maps giving  weak homotopy equivalences on $Y$. The cofibrations are the maps with the fiberwise homotopy extension property (FHEP) over $E$.

\begin{df} The bivariant $A$-theory of a fibration $p$ is defined as  $$\bA(E \overset{p}\to B) := K(R_{hf}(E \overset{p}\to B)).$$
Each pullback square of fibrations is assigned to a map
\[ \xymatrix{ E' \ar[r] \ar[d]_-{p'} & E \ar[d]^-p \\ B' \ar[r] & B } \raisebox{-1.6em}{ $\qquad \leadsto \qquad \bA(E \overset{p}\to B) \to \bA(E' \overset{p'}\to B')$ } \]
using the exact functor that pulls back each space $Y$ along $E' \to E$. This makes $\bA$ into a contravariant functor (see \cite[Rmk 3.5]{raptis_steimle}).
\end{df}

\noindent Note that $\bA$ contains as a special case both Waldhausen's $\bA(X) = \bA(X \to *)$ and the contravariant analog $\uda(X) = \bA(X \overset{\id}\to X)$.

\begin{rem}
This definition of bivariant $A$-theory is equivalent to the one given in \cite{raptis_steimle} by an application of the approximation property. Their cofibrations are the maps having the homotopy extension property (HEP) on each fiber separately.
\end{rem}

We may now prove \autoref{intro_bivariant}. We regard $\bA_G^{\textup{coarse}}(X)$ as a symmetric spectrum obtained by iteration of the $\mc S_{\sbt}$-construction, with $G$-action induced by the $G$-action on $R(X)$. We are therefore only considering its underlying na\"ive $G$-spectrum. 

\begin{prop}\label{prop:coarse_equals_bivariant} There is a natural equivalence of symmetric spectra
\[ \bA_G^{\textup{coarse}}(X)^H \simeq \bA(EG \times_H X \to BH) \]
In particular,
\[ \bA_G^{\textup{coarse}}(X)^{\{e\}} \simeq \bA(X), \qquad \bA_G^{\textup{coarse}}(*)^H \simeq \uda(BH) \]
\end{prop}


\begin{proof}
From \autoref{fixed_points_agree}, the fixed points  $\bA_G^{\textup{coarse}}(X)^H$ are given by the Waldhausen $K$-theory of the category  $R_{hf}(X)^{hH}$, which we identify with the category of retractive $H$-equivariant spaces over $X$ with underlying  homotopy finite space, as in \autoref{rxhh}. As explained in Remark \autoref{rmk:could_be_strong}, we may restrict $R_{hf}(X)$ to the spaces with the homotopy type of relative cell complexes, with strong homotopy equivalences and HEP cofibrations. We do so in this proof.

We adopt the shorthand
\[ E = EG \times_H X = B(*,G,G \times_H X), \qquad B = BH = B(*,G,G/H) \]
In particular, we consider $EG$ to be a right $G$-space, not a left one as we did when defining the category $\tG$. We freely use the result that for a well-based topological group $H$ the map $B(*,H,H) \to B(*,H,*)$ is a principal $H$-bundle \cite[Cor 8.3]{MayClass}. This implies that $B(*,H,X) \to B(*,H,*)$ is a fiber bundle with fiber $X$. Since realization of simplicial spaces commutes with strict pullbacks, our desired map $p\colon E \to B$ is a pullback of this fiber bundle, hence also a fiber bundle.

The equivalence of $K$-theory spectra will be induced by the functor $$\Phi\colon R_{hf}(X)^{hH} \ra R_{hf}(EG \times_H X \overset{p}\to BH)$$ that applies $EG \times_H -$ to the retractive space $(Y,i_Y,p_Y)$ over $X$, obtaining a retractive space over $EG \times_H X$:
\[ \xymatrix @C=5em{ EG \times_H X \ar[r]^-{EG \times_H i_Y} & EG \times_H Y \ar[r]^-{EG \times_H p_Y} & EG \times_H X } \]
The composite map $EG \times_H Y \to BH$ is a fiber bundle with fiber $Y$, which is assumed to be a homotopy finite retractive space over $X$.
Therefore $\Phi$ indeed lands in the Waldhausen category $R_{hf}(EG \times_H X \overset{p}\to BH)$.
It is elementary to check that weak equivalences and cofibrations are preserved, and therefore $\Phi$ induces a map on $K$-theory.

To prove that this map is an equivalence we verify the approximation property from \cite{waldhausen}. We observe that the category $\Cat(\tH, R_{hf}(X))^H$ has a tensoring with unbased simplicial sets sending the $H$-space $Y$ over $X$ and a simplicial set $K$ to the external smash product $Y \barsmash |K|_+$. This has the pushout-product property, by the usual formula for an NDR-pair structure on a product of NDR-pairs. Therefore $\Cat(\tH, R_{hf}(X))^H$ has a cylinder functor.

For the first part of the approximation property, note that the map of bundles $EG\times_H Y\to EG\times_H Y'$ is an equivalence if and only if the map of fibers $Y\to Y'$ is an equivalence. For the second part of the approximation property, we use the right adjoint $F(EG,-)$ of the functor $\Phi$ when regarded as a functor from $H$-equivariant spaces under $X$ to spaces under $\Phi(X) = EG \times_H X$. Given a cofibrant retractive $H$-space $Y$ and a map of retractive $\Phi(X)$-spaces $\Phi(Y) \to Z$, we factor the adjoint into a mapping cylinder
\[ \xymatrix{ Y \ar[r] & Y' = Y \barsmash I_+ \cup_{Y \times 1} F_{BH}(EG,Z) \ar[r]^-\sim & F_{BH}(EG,Z) \ar[r] & F(EG,Z) } \]
The map $Y \to Y'$ is a cofibration of spaces under $X$ and over $F(EG,EG \times_H X)$ by the pushout-product property. Pushing $Y'$ back through the adjunction, we get a factorization of retractive spaces over $\Phi(X)$
\[ \xymatrix{ \Phi(Y) \ar[r] & \Phi(Y') \ar[r]^-\sim & Z } \]
The map $\Phi(Y') \to Z$ is an equivalence because it is a map of fibrations whose induced map of fibers is measured by the equivalence $Y' \to F_{BH}(EG,Z)$ from above. This finishes the proof.
\end{proof}

\section{Transfers on Waldhausen $G$-categories}

In this section, we give the construction of $\bA_G(X)$ and prove the following main theorem (\autoref{thm:agx_exists} from the Introduction.)

\begin{thm}\label{agx_exists2}
For $G$ a finite group, there exists a functor $\bA_G$ from $G$-spaces to genuine $G$-spectra with the property that on $G$-fixed points,
\[ \bA_G(X)^G \simeq \prod_{(H) \leq G} \bA(X^H_{hWH}), \]
and a similar formula for the fixed points of each subgroup $H$.
\end{thm}

We construct $\bA_G(X)$ as a spectral Mackey functor because we need the flexibility to refine the weak equivalences in each of the homotopy fixed point categories $R_{hf}(X)^{hH}$. We describe the framework of spectral Mackey functors as models of $G$-spectra, developed by Guillou and May in \cite{GM1}, followed by the work of Bohmann and Osorno \cite{bohmann_osorno}, which constructs categorical input that directly feeds into their theorem. We then construct this categorical input by a 1-categorical variant of a general construction due to Barwick, Glasman and Shah. In particular, our \autoref{clark} can be viewed as a reinterpretation of \cite[8.1]{Gmonster2}. Finally, we construct $\bA_G(X)$ by descending the structure to the Waldhausen categories with refined weak equivalences for each $H\subseteq G$.

\subsection{Review of spectral Mackey functors}

We start with a description of the framework in broad strokes. By a general result of Schwede and Shipley \cite{schwede_shipley_morita}, if $\sC$ is a stable model category with a finite set of generators $\{X_1,\ldots,X_n\}$, then the derived mapping spectra $\sC(X_i,X_j)$ form a spectrally enriched category $\mc B(\sC)$ on the objects $\{X_1,\ldots,X_n\}$. Thinking of such a spectral category as the many-objects version of a ring spectrum, and spectrally-enriched functors into spectra as modules over that ring, there is a model category $\mc Mod_{\mc B(\sC)}$ of modules over $\mc B(\sC)$ and a Quillen equivalence $\mc Mod_{\mc B(\sC)} \simeq \sC$ given by coend with $\{X_i\}$ and its right adjoint:
\[ L(\{M_i\}) = \{X_i\} \sma_{\mc B(\sC)} \{M_i\}, \qquad R(Y)_i = \sC(X_i,Y). \]
This is the spectral analog of classical Morita theory. When $R$ is a ring and $M$ a perfect $R$-module generator, this construction gives an equivalence between $R$-modules and $\End_R(M)$-modules.

Taking $\sC$ to be the category of orthogonal $G$-spectra for a finite group $G$, $\sC$ is generated by the suspension spectra $\Sigma^\infty_+ G/H$ for conjugacy classes of subgroups $(H) \leq G$. By the self-duality of the orbits $\Sigma^\infty_+ G/H$, the mapping spectrum from $G/H$ to $G/K$ may be written as the genuine fixed points of a suspension spectrum
\[ (\Sigma^\infty_+ G/H \times G/K)^G \]
and the compositions are given by stable $G$-maps
\[ G/H \times G/L \times G/L \times G/K \ra G/H \times G/L \times G/K \ra G/H \times G/K \]
which collapse away the complement of the diagonal of $G/L$ and then fold that diagonal to a single point. This gives a category enriched in orthogonal spectra, or symmetric spectra by neglect of structure.

Guillou and May prove that this category is equivalent to a spectral version of the Burnside category, namely a category $\GB$ enriched in symmetric spectra, with objects $G/H$ and morphism symmetric spectra $\GB(G/H,G/K)$ given by the $K$-theory of the permutative category of finite equivariant spans from $G/H$ to $G/K$.  The composition is by pullback of spans, which can be made strictly associative by using a skeleton of the category of finite $G$-sets and by picking explicit models for pullbacks of spans (cf. \cite{GM1}). 

\begin{thm}[Guillou-May]\label{thm:guillou_may}There is a string of Quillen equivalences between $\GB$-modules $\{M_H\}$ in symmetric spectra and genuine orthogonal $G$-spectra $X$. The underlying symmetric spectrum of the fixed points $X^H$ is equivalent to the spectrum $M_H$ for every subgroup $H$.
\end{thm}

Therefore, by \autoref{thm:guillou_may}, to create a $G$-spectrum whose $H$-fixed points are $K(R^H_{hf}(X))$, it is enough to show that the symmetric spectra $K(R^H_{hf}(X))$ form a module over the ``ring on many objects'' $\GB$. The spectral category $\GB$ from \cite{GM1} is built using permutative categories; following \cite{bohmann_osorno}, we give an alternate version $\GB_{Wald}$ using Waldhausen categories.

\begin{df}\label{def:Wald_cat_of_spans}
For each pair of subgroups $H,K \leq G$ let $S_{H,K}$ denote the category of finite $G$-sets containing $G/H \times G/K$ as a retract. Such sets are of the form $S \amalg (G/H \times G/K)$, which we abbreviate to $S_+$ when $H$ and $K$ are understood. This is a Waldhausen category in which the weak equivalences are isomorphisms and the cofibrations are injective maps. Of course, the coproduct is disjoint union along $G/H \times G/K$. The zero object is the retractive $G$-set $G/H\times G/K$, namely $\emptyset_+$. We note that it is precisely in order to have a zero object and thus a Waldhausen structure, that we need to consider retractive $G$-sets over $G/H \times G/K$ instead of just spans.

We adopt the conventions of  \cite[\S 1.1]{GM1}, assuming that each of the $G$-sets $S$ is one of the standard sets $\{1,\ldots,n\}$ with a $G$-action given by some homomorphism $G\ra \Sigma_n$, so that the coproduct, product, and pullback are given by specific formulas that make them associative on the nose. In particular, the pullback is defined by taking a subset of the product, ordered lexicographically.

Define a pairing
\[ \ast\colon S_{H,L} \times S_{L,K} \ra S_{H,K} \]
by sending each pair of composable spans $S_+ = S \amalg (G/H \times G/L)$ and $T_+ = T \amalg (G/L \times G/K)$ to the span $(S * T)_+ = (S * T) \amalg (G/H \times G/K)$, where $(S * T)$ is the pullback span
\[ \xymatrix @R=1em @C=1em{
&& (S * T) \ar[ld]_-{p_3} \ar[rd]^-{q_3} && \\
& S \ar[ld]_-{p_1} \ar[rd]^-{q_1} && T \ar[ld]_-{p_2} \ar[rd]^-{q_2} & \\
G/H && G/L && G/K. } \]
Notice that $(S*T)_+$ with the basepoint section is a quotient of the pullback of $S_+$ and $T_+$. This allows us to define for each $f\colon S_+\to S'_+$ and $g\colon T_+\to T'_+$ a map $f*g\colon (S*T)_+ \to (S'*T')_+$ by the universal property of the pullback and the quotient.
This pairing is biexact and strictly associative by our adopted conventions. 

\end{df}

\begin{remark}\label{clone_unit} As discussed in  \cite[\S 1.1]{GM1}, the chosen model for the pullback of $G$-sets has the slight defect that the unit span $(G/H)_+  = G/H \amalg (G/H\times G/H)$ with identity projections
\[ \xymatrix @R=1em @C=1em{
& G/H \ar[ld]_-{\id} \ar[rd]^-{\id} & \\
G/H && G/H } \]

\noindent  is not a strict unit on both sides  of the horizontal composition $*$, but only a unit up to canonical isomorphism on the left side. In order to rectify this, one whiskers the category of spans with a new object $\mathbf{1}_{G/H}$ and a unique isomorphism $\mathbf{1}_{G/H} \cong (G/H)_+$,
and then declares that $\mathbf{1}_{G/H}$ acts as a strict unit for $*$. The structure we defined above extends to $\mathbf{1}_{G/H}$. This follows from the coherence condition that the canonical isomorphisms $(G/H)_+ * S \cong S$ and $S \cong S * (G/K)_+$ are natural in $S$, that the two resulting maps from $S * T$ to $(G/H)_+ * S * T$ must coincide, and a similar statement relating $S * T$ to $S * (G/K)_+ * T$ and to $S * T * (G/L)_+$.
\end{remark}

\begin{df}
Let $\GB_{Wald}$ be the spectrally-enriched category on the objects $G/H$, $(H) \leq G$ whose mapping spectra are the Waldhausen $K$-theory spectra $K(S_{H,K})$.
\end{df}

To translate \autoref{thm:guillou_may} into something that interacts more readily with Waldhausen categories, we use the following result. It follows from the main result of \cite{bohmann_osorno}, which gives a comparison of the multiplicative Segal $K$-theory constructed in \cite{EM} and the multiplicative Waldhausen $K$-theory constructed in \cite{blumberg2011derived,zakharevich2014category}.

\begin{thm}[Bohmann-Osorno]\label{thm:bohmann_osorno}
There is an equivalence of spectrally enriched categories $\GB$ and $\GB_{Wald}$.\end{thm}

Since equivalences of spectral categories induce Quillen equivalences on their module categories \cite[6.1]{schwede_shipley_equivalences}, by \autoref{thm:bohmann_osorno}, it is now enough to show that the spectra $K(R^H_{hf}(X))$ form a module over $\GB_{Wald}$. 
This will follow if we define a ``right action" map of spans on the categories $R^H$,
$$\ast\colon R^H\times S_{H,K} \xrightarrow{ }R^K$$
 such that the action map is a bi-exact functor, and the action is associative and unital.
We will now spell out more explicit categorical conditions that will imply this. 

\begin{prop}\label{our conditions}
Suppose we are given
\begin{enumerate}
\item[1.] a Waldhausen category $R^H$ for each $H \leq G$,
\item[2.] an exact functor $(-*S)\colon R^H \to R^K$ for each retractive span $S_+$ in the category $S_{H,K}$,
\item[3.] a natural transformation of functors $\overline{f}\colon (-*S) \Rightarrow (-*S')$ for each map of retractive spans $f\colon S_+ \to S'_+$,
\end{enumerate}
subject to the conditions
\begin{enumerate}
\item[4.] for fixed $A \in R^H$, the assignment $S_+ \mapsto A * S$ defines a functor $S_{H,K} \to R^K$,
\item[5.] we have $A\ast \emptyset \cong *$ and $(A * S) \vee (A * T) \to A * (S \amalg T)$ is an isomorphism in $R^K$ for all spans $S_+, T_+$ ,
\item[6.] the unit span action $(- * \mathbf{1}_{G/H})\colon R^H \to R^H$ is the identity,
\item[7.] if $(S * T)$ is the horizontal composition of $S$ and $T$ as above then $(- * (S * T)) = ((- * S) * T)$ as functors $R^H \to R^K$, and for maps $f\colon S_+\to S'_+$ and $g\colon T_+\to T'_+$, we have an equality $(\phi\ast f)\ast g  =\phi\ast (f\ast g)$. Here, for $f\colon S_+ \to S'_+$, and for a map $\phi\colon Y\to Y'$ in $R^H$, the map $\phi\ast f$ is defined to be either composite $Y\ast S\to Y'\ast S'$ in the commuting diagram we get from point 3 above: \[ \xymatrix{
Y\ast S \ar[d]_-{\phi\ast S} \ar[r]^-{\overline{f}_Y}& Y\ast S'\ar[d]^-{\phi\ast S'} \\
Y' \ast S \ar[r]^-{\overline{f}_{Y'}} & Y'\ast S'.
}\]
\end{enumerate}
%

\noindent Then the spectra $K(R^H)$ form a module over $\GB_{Wald}$, and therefore, also over $\GB$.
\end{prop}

\begin{proof}
By (5) the functor $S_+ \mapsto A * S$ preserves all sums. Observe that every cofibration $S_+ \to T_+$ is a coproduct of the identity of $S_+$ and the map $\emptyset_+ \to (T-S)_+$. Therefore $A * S \to A * T$ is isomorphic to a sum of the identity of $A * S$ and the inclusion of the zero object $0$ $\to A * (T-S)$ which is a cofibration. Of course, equivalences of spans are isomorphisms, which go to isomorphisms in $R^K$. Therefore the pairing $R^H \times S_{H,K} \to R^K$ is exact in the span variable, and it is exact in the $R^H$ variable by condition 2.

To complete the verification of biexactness, note given an inclusion $S_+ \to T_+$ and a cofibration $A \to B$ in $R^H$, the map $A * T \cup_{A * S} B * S \to B * T$ is a pushout of the map $A * (T-S) \to B * (T-S)$, which is a cofibration because $(T-S)$ acts by an exact functor.

Therefore we have biexact pairings $\ast\colon R^H \times S_{H,K} \to R^K$ with strict associativity and unit. We choose distinguished zero objects $\underline{0}$ for each of the categories $R^H$ and $S^{H,K}$ and apply Waldhausen $K$-theory. We then modify the pairings $*$ to strictly preserve these distinguished zero objects: we set $A * \underline{0} = \underline{0} = \underline{0} * A$ and observe that there is a unique way of extending this modified definition to morphisms, preserving the bifunctoriality of the pairing $*$ along with its strict associativity and unit. By the multifunctoriality of Waldhausen $K$-theory (cf. \cite[6.2]{zakharevich2014category}, \cite[2.6]{blumberg2011derived}), these modified pairings then make the spectra $K(R^H)$ into a module over $\GB_{Wald}$.
\end{proof}

In the next section we show how to give such data for $\mc C^{hH}$ when $\mc C$ is any Waldhausen $G$-category.

\subsection{Categorical transfer maps}\label{transfer_section}
Suppose that $\sC$ is a $G$-category with a chosen sum bifunctor $\oplus$ isomorphic to the categorical coproduct $\amalg$. Since $G$ acts through isomorphisms of categories, it preserves $\oplus$ up to canonical isomorphism. Let $f\colon S \ra T$ be a map of finite $G$-sets. As in the previous section, we assume all of our finite $G$-sets come with a total ordering, which does not have to be preserved by $f$.Note that a finite $G$-set $S$ can be regarded as a category with objects the elements of $S$ and only identity morphisms, so the functor category $\Cat(S, \sC)$ is isomorphic to the $S$-indexed product $\prod_S \sC$. We can define a functor
\[ f_! \colon \Cat(S, \sC)\to  \Cat(T, \sC),\]  on objects by
\[ (f_!F)(t) \colon= \bigoplus_{i \in f^{-1}(t)} F(i),  \] or equivalently,
\[  f_! \colon  \prod_S \sC  \to \prod_T \sC\]
\[ (c_1,\  \dots, c_j) \mapsto \left(   \bigoplus_{i \in f^{-1}(1)} c_i, \ \dots,  \bigoplus_{i \in f^{-1}(k) }c_i
\right), \] where $j=|S|$ and $k=|T|$. The action of $f_!$ on morphisms $F \Rightarrow F'$ is clear because $\bigoplus$ is a functor. Note that each set $f^{-1}(t)$ inherits a total ordering, which we use to define the above sum, although changing the ordering would only change the sum up to a canonical isomorphism. Note also that if $f^{-1}(t)$ is empty, then $(f_!F)(t)$ is a zero object in $\sC$.

The functor $f_!$ is not on the nose equivariant, even if the sum $\oplus$ in $\sC$ commutes with the $G$-action strictly. It is only pseudo-equivariant. When we apply the $\Cat(\tG, -)$, by \autoref{pseudo equiv}, we get an on the nose equivariant functor
\[ f_!\colon \Cat(S \times \tG, \sC) \cong\Cat(\tG, \Cat(S, \sC))  \ra  \Cat(\tG, \Cat(T, \sC))\cong \Cat(T \times \tG, \sC), \]which upon taking $G$-fixed points gives a transfer (or pushforward map) along the map of $G$-sets  $f\colon S \ra T$. We make this more explicit in the following definition.

\begin{df}\label{pbpf}
Let $\sC$ be a $G$-category with coproduct $\oplus$, and let $f\colon S \ra T$ be a map of unbased finite $G$-sets. Define a {\bf pullback (restriction) functor} 
\[ f^*\colon \Cat(T \times \tG, \sC)^G \ra \Cat(S \times \tG,\sC)^G \]
on objects $F\colon T \times \tG \to \sC$ by the formulas
\[ (f^*F)(s,g) = F(f(s),g) \]
\[ (f^*F)(s,g \ra h) = F(f(s),g \ra h) \]
and on maps $\alpha\colon F \Rightarrow F'$ by the formula
\[ (f^*F)(s,g) = F(f(s),g) \overset{\alpha}\to F'(f(s),g) = (f^*F')(s,g). \]
Define a {\bf pushforward (transfer) functor}
\[ f_!\colon \Cat(S \times \tG,\sC)^G \ra \Cat(T \times \tG,\sC)^G \]
on objects by
\[ (f_!F)(t,g) := g\left( \bigoplus_{i \in f^{-1}(g^{-1}t)} F(i,1) \right). \]
To finish defining it on objects and morphisms, we use the canonical isomorphism
\[ g\left( \bigoplus_{i \in f^{-1}(g^{-1}t)} F(i,1) \right) \cong \bigoplus_{i \in f^{-1}(g^{-1}t)} F(gi,g) \cong \bigoplus_{j \in f^{-1}(t)} F(j,g). \]
Under this isomorphism, the morphism $(f_!F)(t,g \ra h)$ is chosen to be the coproduct
\[ \bigoplus_{j \in f^{-1}(t)} F(j,g \ra h) \]
and the morphism $(f_!F)(t,g) \to (f_!F')(t,g)$ induced by a map $\alpha\colon F \Rightarrow F'$ is chosen to be the coproduct
\[ \bigoplus_{j \in f^{-1}(t)} \left( F(j,g) \overset\alpha\to F'(j,g) \right). \]
\end{df}

\begin{remark}\label{fsharprem}
We note the following properties of $f_!$ which we will use later on:
\begin{enumerate}
\item If $f$ is an isomorphism, then $f_!=(f^{-1})^\ast;$
\item If $f=\id$, then $\id_!=\id$;
\item if $f$ and $h$ are composable maps of $G$-sets, $(h f)_!\cong f_!h_!$. 
\end{enumerate}
\end{remark}

\begin{remark}
In the special case where $H$ is a subgroup of $K$ and $f\colon G/H\ra G/K$ is the quotient map, $f_!$ defines a transfer map
\[ \sC^{hH} \ra \sC^{hK}. \]
\end{remark}
\noindent More generally, for a span
\[ \xymatrix @R=1em @C=1em{
& S \ar[ld]_-{p} \ar[rd]^-{q} & \\
G/H && G/K } \]
one can define a functor $(-)*S\colon \sC^{hH} \to \sC^{hK}$ by $q_!p^*$.
To prove that $*$ defines a bifunctor that respects compositions of spans, one needs the following formal properties of $f_!$ and $f^*$ (cf. \cite[\S 10]{Gmonster}).

\begin{prop}\label{adj}
For each equivariant map $f\colon S \to T$ of finite $G$-sets, the functors $(f_!,f^*)$ form an adjoint pair.
\end{prop}

\begin{proof}
Let $F\colon S\times \tG \to \sC$ and $F'\colon T\times \tG \to \sC$. Under the canonical isomorphism from the above definition, each transformation $f_!F \Rightarrow F'$ is given by the data of maps
\[ \bigoplus_{s \in f^{-1}(t)} F(s,g) \ra F'(t,g) \]
for each $t\in T$ and $g \in G$. The universal property of $\oplus$ gives a bijection between such collections of maps and collections of maps 
\[ F(s,g) \ra F'(f(s),g) \]
for each $s\in S$ and $g \in G$. This gives the bijection between transformations $f_!F \Rightarrow F'$ and $F\Rightarrow f^\ast F'$.
\end{proof}

\begin{prop}\label{beck-chevalley}
Given a pullback square of finite $G$-sets
\[ \xymatrix @R=1.2em @C=1.2em{
A \ar[r]^-k \ar[d]_-h & B \ar[d]^-f \\
C \ar[r]_-j & D } \]
there is a ``Beck-Chevalley'' isomorphism
\[ \xymatrix{ \Cat(B \times \tG, \sC)^G \ar@/^1em/[rr]^-{h_!k^*} \ar@/_1em/[rr]_-{j^*f_!} \ar@{}[rr]|-{\Downarrow BC} & & \Cat(C \times \tG, \sC)^G } \]
defined as the composite of unit and counit maps
\[ \xymatrix @R=.4em{ h_!k^* \ar[r]^-\eta & j^*j_!h_!k^* \ar[r]^-\cong & j^*f_!k_!k^* \ar[r]^-\epsilon & j^*f_! .} \]
\end{prop}

\begin{proof}
Unwinding the definitions gives a natural transformation between the two functors on $C \times \tG$ defined by
\[ h_!k^*F(c,g) = g\left( \bigoplus_{a \in h^{-1}(g^{-1}c)} F(k(a),1) \right), \quad
j^*f_!F(c,g) = g\left( \bigoplus_{b \in f^{-1}(g^{-1}j(c))} F(b,1) \right). \]
that sends each $F(k(a),1)$ to the $F(b,1)$ where $b = k(a)$, by an identity map. Since the square is a pullback, $k$ defines a bijection $h^{-1}(c) \to f^{-1}(j(c))$ for all $c \in C$, so this is a natural isomorphism.
\end{proof}

\begin{prop}\label{oreo_cookie}
Each diagram of G-sets
\[ \xymatrix @R=.8em {
& S \ar[dl]_-p \ar[dr]^-q \ar[dd]^-f & \\
U && V \\
& T \ar[ul]^-r \ar[ur]_-s & } \]
induces a natural transformation $f_\sharp\colon q_!p^* \ra s_!r^*$. These natural transformations depend in a functorial way on the maps $f$.
\end{prop}

\begin{proof}
\[ \xymatrix @R=.8em {
& \Cat(S \times \tG, \sC)^G \ar@{<-}[dl]_-{p^*} \ar[dr]^-{q_!} \ar@/^/[dd]^-{f_!} & \\
\Cat(U \times \tG, \sC)^G && \Cat(V \times \tG, \sC)^G. \\
& \Cat(T \times \tG, \sC)^G \ar@{<-}[ul]^-{r^*} \ar[ur]_-{s_!} \ar@/^/[uu]^-{f^*} & } \]
The identities $p = rf$, $q = sf$ and the counit of $(f_!,f^*)$ gives a natural transformation
\[ q_!p^* \cong  s_!f_!f^*r^* \overset\epsilon\ra s_!r^* \]
which we take as the definition of $f_\sharp$. Functoriality follows from an easy diagram chase. 
\end{proof}

We conclude this section by checking that for each Waldhausen $G$-category $\mc C$, the action of spans on the categories $\mc C^{hH}$ extends to an action of the categories of retractive spans $S_{H,K}$, giving a spectral Mackey functor in the sense of the previous section. As mentioned earlier, this argument is a strictified analog of the unfurling construction of \cite[\S 11]{Gmonster}.

\begin{prop}\label{clark}(cf. \cite[8.1]{Gmonster2})
Let $\sC$ be a Waldhausen $G$-category. Then the collection of spectra $K(\sC^{hH})$ may be modified up to equivalence to form a module over $\GB$.
\end{prop}

\begin{proof} By \autoref{our conditions}, it suffices to check the following seven points.

\textbf{1.} Set $R^H = \Cat(G/H \times \tG,\mc C)^G \cong \mc C^{hH}$. Recall that this is a Waldhausen category by \autoref{htpyfixedpts}. In order to make the action of  spans strictly associative and unital in steps 6 and 7, we need to thicken this category in the following way. Define a new category $\underline{R}^H$ whose objects are triples $(J, Y,(S_+,p,q))$, where $J \leq G$, $Y$ is an object of $R^J = \Cat(G/J\times\tG,\mc C)^G$, and $S_+$ is a retractive span from $G/J$ to $G/H$ (though we exclude the unit $\mathbf{1}_{G/H}$). To each such triple $(J,Y,S_+)$ we can assign the object $Y * S = q_!p^*Y$ of $R^H$. Then we define the morphisms in $\underline{R}^H$ as
\[ \underline{R}^H((J,Y,S_+),(J',Y',S_+')) \colon= R^H(Y * S, Y' * S'). \]
There is an essentially surjective functor $ \underline{R}^H\to R^H$ which sends an object $(J, Y,(S_+,p,q))$ to $Y * S$. By definition of the morphisms in $\underline{R}^H$, this functor is full and faithful, thus $\underline{R}^H \to R^H$ is an equivalence of categories. We lift the Waldhausen structure of $R^H$ to $\underline{R}^H$ along this equivalence.\\

\textbf{2.} Each span $T_+ \in S_{H,K}$ with maps $G/H \overset{r}\leftarrow T \overset{s}\rightarrow G/K$ defines a functor $(-)*T\colon R^H \to R^K$ by $s_!r^*$. Exactness of $(-)*T$ follows because the coproduct $\oplus$ in $\mc C$ commutes with colimits and preserves both cofibrations and weak equivalences.
On the thickened categories, we define the action map $(-)*T\colon \underline{R}^H \to \underline{R}^K$ on objects by
\[ (J,Y,S_+) * T_+ \colon= (J,Y,(S * T)_+). \] 
In order to extend this definition of the action map on morphisms, recall that the Beck-Chevalley isomorphism for the pullback square in the diagram below gives an isomorphism $Y * (S * T) \cong (Y * S) * T$ of objects of $R^K$:
\[ \xymatrix @R=1em @C=1em{
&& S * T \ar[ld] \ar[rd] & \\
& S \ar[ld]_-{p} \ar[rd]^-{q} && T \ar[ld]_-{r} \ar[rd]^-{s} & \\
G/J && G/H && G/K. } \]
Using these isomorphisms and the isomorphisms from \autoref{fsharprem}, part (3), we define the action of $T$ on a morphism in $\underline{R}^H$. Each morphism $\underline{\phi}\colon (J,Y,S_+) \to (J',Y',S'_+)$ is represented by a morphism $\phi\colon Y * S \to Y' * S'$ in $R^H$. We take it to the composite
\begin{equation}\label{beck_chevalley_sandwich}
\xymatrix{ 
Y * (S * T) \ar[r]^-\cong & (Y * S) * T \ar[d]^-{\phi * T = s_!r^*(\phi)} \\
Y' * (S' * T) \ar[r]^-\cong & (Y' * S') * T .}
\end{equation}
By definition this gives a morphism $$\underline{\phi} * T\colon (J,Y,(S * T)_+) \to (J',Y',(S' * T)_+).$$ By pasting two diagrams of the form \eqref{beck_chevalley_sandwich} together, we see this respects composition and units, and so defines a functor $\underline{R}^H \to \underline{R}^K$.  Finally, when $H = K$ we define $\mathbf{1}_{G/H}$ to act as the identity functor of $\underline{R}^H$. Note that by construction, the diagram
\[\xymatrix @R=3em @C=3em{
\underline{R}^H \ar[d] \ar[r]^-{(-)\ast T} & \underline{R}^K\ar[d] \\
R^H \ar[r]_-{(-)\ast T} & R^K 
}\] commutes up to natural isomorphism. Moreover, the top map is exact because the bottom one is.\\


\textbf{3.} Given a map of retractive spans
\[ \xymatrix  @R=.8em @C=4.5em{
& T \amalg (G/H \times G/K) \ar[dl]_-{r \amalg \pi_1} \ar[dr]^-{s \amalg \pi_2} \ar[dd]^-f & \\
G/H && G/K \\
& T' \amalg (G/H \times G/K) \ar[ul]^-{r' \amalg \pi_1} \ar[ur]_-{s' \amalg \pi_2} & } \]
we recognize canonical isomorphisms
\[ (s \amalg \pi_2)_!(r \amalg \pi_1)^* \cong s_!r^* \vee (\pi_2)_!\pi_1^*. \]
We can then define the component of the natural transformation $\overline{f}_Y\colon Y\ast T\to Y\ast T' $ to be the summand of $f_\sharp$ from \autoref{oreo_cookie} taking $s_!r^*$ to $s'_!r'^*$. Note that because $f$ restricts to the identity of $G/H \times G/K$, we have the commuting diagram
\[ \xymatrix{
(\pi_2)_!\pi_1^* \ar[r] \ar@{=}[d] & s_!r^* \vee (\pi_2)_!\pi_1^* \ar[r] \ar[d]^-{f_\sharp} & s_!r^* \ar[d]^-{\overline{f}} \\
(\pi_2)_!\pi_1^* \ar[r] & s'_!r'^* \vee (\pi_2)_!\pi_1^* \ar[r] & s'_!r'^*.
} \]
For each $(J,Y,S_+)$ in the thickening $\underline{R}^H$, this defines a map in $R^K$ from $(Y * S) * T$ to $(Y * S) * T'$. We use the Beck-Chevalley isomorphisms as in \eqref{beck_chevalley_sandwich} to lift this to a map $\overline f$ in $\underline{R}^K$. The verification that $\overline f$ is a natural transformation of functors $\underline{R}^H \to \underline{R}^K$ quickly reduces to $R^H \to R^K$, which can be proven using the diagram just above.

To handle the case where one of $T$ or $T'$ is the unit $\mathbf{1}_{G/H}$, we use the canonical isomorphism $Y * S \cong (Y * S) * G/H$ in the place of the Beck-Chevalley isomorphism in the diagram \eqref{beck_chevalley_sandwich}.\\

\textbf{4.} As in the previous point, the claim that the maps $\overline{f}$ respect composition on the categories $\underline{R}^H$ quickly reduces to the categories $R^H$. Given two maps of spans
\[ \xymatrix @R=1.8em @C=6em{
& S \amalg (G/H \times G/K) \ar[dl]_-{p \amalg \pi_1} \ar[dr]^-{q \amalg \pi_2} \ar[d]^-f & \\
G/H & T \amalg (G/H \times G/K) \ar[l]_{r \amalg \pi_1} \ar[r]^{s \amalg \pi_2} \ar[d]^-h & G/K,  \\
& U \amalg (G/H \times G/K) \ar[ul]^-{m \amalg \pi_1} \ar[ur]_-{n \amalg \pi_2} & } \]
\autoref{oreo_cookie} tells us that $(hf)_\sharp = h_\sharp f_\sharp$. A simple chase of the diagram below confirms that $\overline{hf} = \overline{h} \circ \overline{f}$:
\[ \xymatrix{
q_!p^* \vee (\pi_2)_!\pi_1^* \ar[r] \ar[d]^-{f_\sharp} \ar@/_4em/[dd]_-{(hf)_\sharp} & q_!p^* \ar[d]^-{\overline{f}} \\
s_!r^* \vee (\pi_2)_!\pi_1^* \ar[r] \ar[d]^-{h_\sharp} & s_!r^* \ar[d]^-{\overline{h}} \\
m_!n^* \vee (\pi_2)_!\pi_1^* \ar[r] & m_!n^*.
} \]
\\

\textbf{5.} Again the claim is equally true for $\underline{R}^H$ and $R^H$. If $\iota$ is the inclusion of the empty set then $\iota_!$ always gives a zero object. The isomorphism $A * S \vee A * T \to A * (S \amalg T)$ is immediate from the definition of the transfer $q_!$.\\

\textbf{6.} This is automatic from the definition in point 2 above.\\

\textbf{7.}  We note that this property holds only for $\underline{R}^H$, not $R^H$; the action of $S_{H,K}$ on the objects of $R^H$ is not strictly associative. However, the action of the objects of $S_{H,K}$ on the objects of $\underline{R}^H$ is associative because we chose a model for spans whose compositions were strictly associative.  The morphisms are more subtle. If we have $\underline{\phi}\colon (J,Y,S_+) \to (J',Y',S'_+)$ and maps of spans $f\colon T \to T'$ and $g\colon U \to U'$, we need to show that we have an equality of maps
\[ (\underline\phi * f) * g = \underline\phi * (f * g)\colon (J,Y,(S*T*U)_+) \ra (J',Y',(S'*T'*U')_+). \]
Once we prove this, associativity on morphisms will also hold automatically in the case where one or more of $T$, $T'$, $U$, or $U'$ is a strict unit $\mathbf{1}_{G/-}$.

\[ \xymatrix @R=1em @!C=1em{
&&& S * T * U \ar[ld] \ar[rd] && \\
&& S * T \ar[ld] \ar[rd] && T * U \ar[ld] \ar[rd] & \\
& S \ar[ld] \ar[rd] && T \ar[ld] \ar[rd] && U \ar[ld] \ar[rd] & \\
G/J && G/H && G/K && G/L } \]
Consider the diagram below, in which the horizontal maps are Beck-Chevalley isomorphisms arising from the pullback squares in the diagram above. 
\[ \xymatrix@R=2em @C=0.9em{
Y * (S * T * U) \ar[r]^-\cong \ar@{-->}[d] & (Y * (S * T)) * U \ar[r]^-\cong & ((Y * S) * T) * U \ar[d]_-{(\phi * f) * g} & (Y * S) * (T * U) \ar[l]_-\cong \ar[d]^-{\phi * (f * g)} \\
Y' * (S' * T' * U') \ar[r]^-\cong & (Y' * (S' * T')) * U' \ar[r]^-\cong & ((Y' * S') * T') * U' & (Y' * S') * (T' * U') \ar[l]_-\cong
} \]
If the dotted map is chosen to make the left-hand rectangle commute, then it defines $(\underline{\phi} * f) * g$. The composite along the entire top row is a Beck-Chevalley map, by the standard fact that they agree along pasting pullback squares. Therefore if the dotted map is chosen to make the outside rectangle commute it defines $\underline{\phi} * (f * g)$.

It therefore suffices to prove that the right-hand square commutes. We expand it in the following way, where $X=Y\ast S$, $X'=Y'\ast S'$, and the vertical maps are Beck-Chevalley isomorphisms:
\[\xymatrix@R=2em @C=1.7em{
X\ast (T\ast U) \ar[rr]^-{\phi\ast (T\ast  U)} \ar[d]_-\cong && X' \ast (T\ast U) \ar@/^1em/[rrrr]^-{\overline{(f\ast g)}_{X'}} \ar[rr]_-{\overline{(1*g)}_{X'}} \ar[d]^-\cong && X' \ast (T\ast U') \ar[d]^-\cong \ar[rr]_-{\overline{(f*1)}_{X'}} && X'\ast(T'\ast U') \ar[d]^-\cong \\
(X\ast T)\ast U \ar[rr]_-{(\phi\ast T)\ast U} && (X'\ast T)\ast U \ar[rr]_-{\overline{g}_{(X'\ast T)}} && (X'\ast T)\ast U' \ar[rr]_-{\overline{f}_{X'}\ast U'} && (X'\ast T')\ast U' \\
}\]
The left-hand square commutes by the naturality of the Beck-Chevalley isomorphism. Each of the last squares is proven formally by a long diagram-chase, or more easily by writing the explicit formula for the two natural transformations and verifying that they are the same direct sum of identity maps and zero maps.

\end{proof}

\subsection{Construction of $\bA_G(X)$}\label{sec:final}

By the previous section, the spectra $K(R_{hf}(X)^{hH})$ form a spectral Mackey functor. To construct $\bA_G(X)$ we simply need to check that the structure thus defined on $R_{hf}(X)^{hH}$ respects the subcategory of retractive $H$-cell complexes and the equivariant weak equivalences and cofibrations between them.

\begin{df} Let $R^H(X)$ be the category of 
$H$-equivariant retractive spaces $Y$ over $X$ with $H$-equivariant inclusion $i_Y$ and retraction $p_Y$. The morphisms are $H$-equivariant maps between these.
The weak equivalences are those inducing weak equivalences rel $X$ on the fixed points for all subgroups of $H$. The cofibrations are the maps $Y \ra Z$ with the $H$-equivariant FHEP: there is an $H$-equivariant, fiberwise retract
\[ Z \times I \ra Y \times I \cup_{Y \times 1} Z \times 1. \]
In particular, when $L \leq H$, the $L$-fixed points of a cofibration are a cofibration in $R(X^L)$. Finally, let $R^H_{hf}(X)$ be the subcategory of objects which are 
retracts in the homotopy category of $R^H(X)$ of finite relative $H$-cell complexes $X \ra Y$. 
\end{df}

\begin{remark}
	As categories, we have an equivalence $R^H(X) \simeq R(X)^{hH}$ by \autoref{rxhh}. But $R^H(X)$ has fewer cofibrations and weak equivalences. The subcategories $R^H_{hf}(X)$ and $R_{hf}(X)^{hH}$ are also distinct -- the first one is defined using finite $H$-cell complexes, the second defined using spaces whose underlying nonequivariant space is a finite cell complex. These differences are the reason why $\bA_G(X)$ and $\bA_G^{\textup{coarse}}(X)$ are not equivalent.
\end{remark}

We want to define an action of spans on $R^H_{hf}(X)$. From the previous section, each span $S$ over $G/H$ and $G/K$ already acts on the larger category $R(X)^{hH}$. It therefore suffices to check that these actions are exact with respect to the more refined Waldhausen structure coming from $R^H(X)$, and preserve the more restrictive finiteness condition that defines $R^H_{hf}(X)$.

\begin{prop}\label{all the actual work}
The functor $(-*S)\colon R(X)^{hH} \to R(X)^{hK}$ restricts to an exact functor $R^H_{hf}(X) \to R^K_{hf}(X)$.
\end{prop}

\begin{proof}
Let \[ \xymatrix @R=1em @C=1em{
& S \ar[ld]_-{p} \ar[rd]^-{q} & \\
G/H && G/K } \] be a given span. From the definition of $q_!p^\ast$ it is clear that up to isomorphism the resulting retractive space over $X$ is a coproduct of the spaces one would get from considering each orbit of $S$ separately. Therefore, without loss of generality, we assume $S \cong G/L$. Recall that the $G$-maps $p\colon G/L\to G/H$ and $q\colon G/L\to G/K$ exist if and only if  $L$ is subconjugate to $H$ and $K$, i.e., they are composites of subgroup inclusions and isomorphisms. Also, recall from \autoref{fsharprem} (1) that if $f$ is an isomorphism, then $f_!=(f^{-1})^\ast$. So it is enough to show:

\begin{enumerate}
\item if $L \leq H$ is a subgroup, the pullback of $p \colon G/L\to G/H$  gives an exact functor $p^\ast\colon R_{hf}^H(X) \ra R_{hf}^L(X)$,
\item if $L$ and $L'$ are conjugate by $L' = gLg^{-1}$ the pullback of the isomorphism $f\colon G/L \xrightarrow{\cong} G/L'$ gives and exact functor $f^\ast\colon R_{hf}^{L'}(X) \ra R_{hf}^{L}(X)$, 
\item if $L \leq K$ is a subgroup, the pushforward of $q\colon G/L\to G/K$ gives an exact functor $q_!\colon R_{hf}^{L}(X) \ra R_{hf}^K(X)$.
\end{enumerate}

\noindent To show (1), suppose $L \leq H$, and let $p\colon G/L\to G/H$ be the map $gL\mapsto gH$. It will suffice to describe a functor $p^\ast\colon R^H(X) \to R^L(X)$ making the square of functors
\[ \xymatrix{
	R^H(X) \ar@{<-}[d]_-\sim \ar@{-->}[r]^-{p^\ast} & R^L(X) \ar@{<-}[d]^-\sim \\
	Cat(G/H \times \tG,R(X))^G \ar[r]^-{p^\ast} & Cat(G/L \times \tG,R(X))^G
} \]
commute up to isomorphism, and show that this $p^\ast$ is exact and preserves the finite complexes. An $H$-equivariant retractive space $(Y,i_Y,p_Y)$ in the top-left comes from a $G$-equivariant functor $F\colon G/H\times \tG\to R(X)$ for which $F(eH,e) = (Y,i_Y,p_Y)$ and $\phi_h$ is the action of $h^{-1}$. When this is restricted to $G/L \times \tG$, it sends $(eL,e)$ to $(Y,i_Y,p_Y)$ and $\phi_{\ell}$ is the action of $\ell^{-1}$. Clearly, we can set $p^\ast$ to be the functor that restricts the $H$-action to the action of $L$, and this makes the above square commute (on the nose). Now we can easily see that this preserves the cofibrations and weak equivalences in $R^H$. Since all groups are finite, it also preserves finite complexes, so it respects the subcategories $R^H_{hf}(X)$ and $R^L_{hf}(X)$.

For (2) consider $L' = gLg^{-1}$, where we fix a choice of such $g$ from all the choices related by conjugation in $L$. Let $f\colon G/L \congar G/L'$ be the isomorphism of $G$-sets given by $hL \mapsto hg^{-1}L'$.  As before, we choose a functor $f^\ast$ making the diagram
\[ \xymatrix{
	R^{L'}(X) \ar@{<-}[d]_-\sim \ar@{-->}[r]^-{f^\ast} & R^L(X) \ar@{<-}[d]^-\sim \\
	Cat(G/L' \times \tG,R(X))^G \ar[r]^-{f^\ast} & Cat(G/L \times \tG,R(X))^G
} \]
commute up to isomorphism. We choose the functor that sends the $L'$-equivariant retractive space $(Y,i_Y,p_Y)$ to the retractive space $(Y,i_Y \circ g,g^{-1} \circ p_Y)$, with each element $\ell \in L$ acting on $Y$ by the given action of $g\ell g^{-1} \in L'$. The commuting diagram
\[ \xymatrix{
	X \ar[r]^-g \ar[d]^-{\ell} & X \ar[r]^{i_Y} \ar[d]^-{g\ell g^{-1}} & Y \ar[r]^{p_Y} \ar[d]^-{g\ell g^{-1}} & X \ar[r]^-{g^{-1}} \ar[d]^-{g\ell g^{-1}} & X  \ar[d]^-{\ell} \\
	X \ar[r]_-g & X \ar[r]_{i_Y} & Y \ar[r]_{p_Y} & X \ar[r]_-{g^{-1}} & X
} \]
demonstrates that this action indeed respects the existing action of $L \leq G$ on $X$. This clearly gives a functor that preserves cofibrations, weak equivalences and finite cell complexes.

It suffices to show that this definition of $f^\ast$ agrees with the original one along the above equivalences of categories. To do this we first modify the original $f^\ast$ up to isomorphism.
We observe that the map $-g^{-1}\colon \ti G \ra \ti G$ that multiplies on the right by $g^{-1}$ is a $G$-equivariant isomorphism of categories, and that any $G$-equivariant functor $\Phi\colon \ti G \ra \mathcal C$ is $G$-equivariantly isomorphic to $\bar{\Phi} = \Phi \circ -g^{-1}$. The components of the natural transformation that give the isomorphism $\Phi\Rightarrow \bar{\Phi} $ are just $\Phi$ applied to the unique isomorphisms $g_0 \xrightarrow{\cong} g_0g^{-1}$ in $\tG$.  

Replace the original $f^\ast$ by the composition of this operation and $f^\ast$. Then, if we start with a functor $F\in \Cat(G/L' \times \tG,R(X))^G$ whose image in $R^{L'}(X)$ is $(Y,i_Y,p_Y)$, this modified pullback of $F$ gives the retractive space
$$(f^\ast \bar{F})(eL, e)=(f^\ast F)(eL, g^{-1})=F(g^{-1}L', g^{-1})=g^{-1}F(eL', e).$$
which is precisely $(Y,i_Y \circ g,g^{-1} \circ p_Y).$ The action of $\ell \in L$ given by the morphism

$$(f^\ast \bar{F})(eL, \ell \ra e)=(f^\ast F)(eL, \ell g^{-1} \ra g^{-1})=F(g^{-1}L', \ell g^{-1} \ra g^{-1})$$
$$=g^{-1}F(eL', g\ell g^{-1} \ra e).$$
Recalling that the $g^{-1}$ on the outside acts trivially on the map on $Y$, this morphism must be $\phi_{g\ell g^{-1}}^{-1}$, in other words the original action of $g\ell g^{-1}$ on $Y$. Therefore our square of functors relating the two definitions of $f^\ast$ commutes strictly (after we modified the bottom map up to isomorphism). Note that different choices of $g$ in this argument produce isomorphic functors, so $f^\ast$ is isomorphic to any of the exact functors obtained by any initial choice of $g$ with the property that $L' = gLg^{-1}$.

Finally, for (3) consider $L \leq K$. Since the pushforward along $q\colon G/L \ra G/K$ is the left adjoint to the pullback, and left adjoints are unique up to natural isomorphism, it must induce on $R^L(X) \ra R^K(X)$ the left adjoint to the forgetful functor $q^\ast$ which restricts the group action from $K$ to $L$. On each retractive $L$-equivariant space $Y$, this left adjoint $q_!Y$ is naturally isomorphic to the pushout
\[ \xymatrix{
K \times_L Y \ar[r] & q_!Y \\
K \times_L X \ar[r] \ar[u] & X. \ar[u] } \]
\noindent We recall that if $H \leq K$ then the $H$-fixed points of $K \times_L Y$ can be computed as
\[ (K \times_L Y)^H \cong \coprod_{\{kL\in K/L \ | \  k^{-1} H k \ \leq \ L\}} Y^{k^{-1} H k}. \]

Since fixed points commute with pushouts along a closed inclusion, we get the pushout square
\[ \xymatrix{
\coprod_{\{kL\in K/L \ | \  k^{-1} H k \ \leq \ L\}} Y^{k^{-1} H k} \ar[r] & (q_!Y)^H \\
\coprod_{\{kL\in K/L \ | \  k^{-1} H k \ \leq \ L\}} X^{k^{-1} H k} \ar[r] \ar[u] & X^H. \ar[u] } \]
From this it is clear that if $Y \ra Z$ is a map of $L$-spaces giving an equivalence on all fixed points, it induces an equivalence of pushouts. Similarly, these constructions all commute up to isomorphism with mapping cylinder, so this construction preserves cofibrations.
Finally we check that it preserves finite complexes by an induction on the number of cells.
For the base case, we observe that if $N \leq L$ is any subgroup, $X \amalg (L/N \times D^n)$ is sent to $X \amalg (K/N \times D^n)$, and similarly with $S^{n-1}$ in the place of $D^n$. Therefore cells are sent to cells.
For the inductive step, we observe that each cell attaching diagram is sent to a cell attaching diagram, because by exactness the pushouts along cofibrations are preserved. Thus the pushforward of $G/L\to G/K$ gives an exact functor $q_!\colon R_{hf}^{L}(X) \ra R_{hf}^K(X)$.

\end{proof}

This establishes the first two conditions from \autoref{our conditions}. The remaining five conditions automatically descend from $R(X)^{hH}$ to any full subcategory with the same coproducts. Therefore the spectra $K(R^H_{hf}(X))$ form a spectral Mackey functor, so there exists a $G$-spectrum $\bA_G(X)$ whose fixed points are $\bA_G(X)^H\simeq K(R^H_{hf}(X))$. It has been long known that  the $K$-theory of the Waldhausen category $R^H_{hf}(X)$ has a splitting  $$ K(R^H_{hf}(X))\simeq \prod_{(H) \leq G} \bA(X^H_{hWH}).$$ A proof of this can be found in \cite{wojciech}. Therefore, we can conclude that the fixed points of the genuine $G$-spectrum $\bA_G(X)$ have a tom Dieck type splitting:

\[ \bA_G(X)^H \simeq \prod_{(H) \leq G} \bA(X^H_{hWH}), \]
This finishes the proof of \autoref{agx_exists2}.

  \bibliographystyle{amsalpha}
  \bibliography{references}

\begingroup%
\setlength{\parskip}{\storeparskip}

\end{document}